\documentclass[a4paper,twoside]{amsart}
\usepackage[utf8]{inputenc}
\usepackage[hidelinks]{hyperref}
\usepackage{tikz}
\usepackage{tikz-cd}
\usetikzlibrary{shapes,arrows}
\tikzstyle{block}=[draw opacity=0.7,line width=1.4cm]
\usepackage{graphicx}
\usepackage{enumitem}
\usepackage{mathrsfs}
\usepackage{etoolbox}
\usepackage{amssymb}
\usepackage{esvect}
\DeclareRobustCommand{\SkipTocEntry}[5]{}
\def\ACl{\mathop{\mathrm{ACl}}\nolimits}

\def\dom{\mathop{\mathrm{Dom}}\nolimits}
\def\range{\mathop{\mathrm{Range}}\nolimits}
\def\Forb{\mathop{\mathrm{Forb_{he}}}\nolimits}
\def\Aut{\mathop{\mathrm{Aut}}\nolimits}
\def\cl{\mathop{\mathrm{Cl}}\nolimits}
\def\F{{\mathcal F}}
\def\Str{\mathop{\mathrm{Str}}\nolimits}
\def\OStr{\mathop{\vv{\mathrm{Str}}}\nolimits}
\def\K{{\mathcal K}}
\def\Age{\mathop{\mathrm{Age}}\nolimits}
\newcommand{\ie}{i.\,e.}

\newcommand{\eg}{e.\,g.}

\def\str#1{\mathbf {#1}}
\def\arity#1{a(\rel{}{#1})}
\def\arityf#1{a(\func{}{#1})}
\def\nbrel#1#2{R\ifstrempty{#1}{}{_{#1}}\ifstrempty{#2}{}{^{#2}}}
\def\rel#1#2{\nbrel{\ifstrempty{#1}{}{\str{#1}}}{#2}}
\def\nbfunc#1#2{F\ifstrempty{#1}{}{_{#1}}\ifstrempty{#2}{}{^{#2}}}
\def\func#1#2{\nbfunc{\ifstrempty{#1}{}{\str{#1}}}{#2}}
\def\permrel#1#2#3{#1(R\ifstrempty{#2}{}{}\ifstrempty{#3}{}{^{#3}})_{\mathbf{#2}}}

\def\permfunc#1#2#3{#1(F\ifstrempty{#2}{}{\ifstrempty{#3}{}{^{#3}})_{\mathbf{#2}}}}
\def\powerset#1{\mathscr{P}(#1)}

\theoremstyle{plain}
\newtheorem{theorem}{Theorem}[section]

\newtheorem{prop}[theorem]{Proposition}
\newtheorem{observation}[theorem]{Observation}

\theoremstyle{definition}
\newtheorem{definition}[theorem]{Definition}
\newtheorem*{example}{Example}
\newtheorem*{remark*}{Remark}
\newtheorem*{claim*}{Claim}
\theoremstyle{remark}
\newtheorem{remark}[theorem]{Remark}

\def\dom{\mathop{\mathrm{Dom}}\nolimits}

\def\str#1{\mathbf {#1}}
\def\Fraisse{Fra\"{\i}ss\' e}

\def\GammaL{\Gamma\!_L}

\begin{document}
\title[Structural Ramsey Theory and EPPA]{Structural Ramsey Theory\\ and\\ the Extension Property for Partial Automorphisms}

\author[J. Hubi\v cka]{Jan Hubi\v cka}
\address{Department of Applied Mathematics (KAM), Charles University, Ma\-lo\-stransk\'e n\'a\-m\v es\-t\'\i~25, Praha~1, Czech Republic}
\curraddr{}
\email{hubicka@kam.mff.cuni.cz}
\thanks{The introduction chapter from habilitation thesis, submitted September 30 2019.
 This paper is part of a project that has received funding from the European Research Council (ERC) under the European Union’s Horizon 2020 research and innovation programme (grant agreement No 810115). Jan Hubi\v cka is further supported by project 18-13685Y of the Czech Science Foundation (GA\v CR), by Center for Foundations of Modern Computer Science (Charles University project UNCE/SCI/004) and by the PRIMUS/17/SCI/3 project of Charles University.}
\maketitle
\begin{abstract}
We survey recent developments concerning two properties of classes of finite structures: the Ramsey property and the extension property for partial automorphisms (EPPA).
\end{abstract}

\section{Introduction}
We survey recent developments concerning two  properties
of classes of finite structures (introduced formally in Section~\ref{sec:preliminaries}).

\medskip

\paragraph{1.}
 Class $\K$ of finite structures is \emph{Ramsey} (or has the \emph{Ramsey property}) if for every $\str{A}, \str{B}\in \K$ there
exists $\str{C}\in \K$ such that for every colouring of the embeddings $\str{A}\to\str{C}$
with two colours there exists an embedding $f\colon\str{B}\to \str{C}$ such that all embeddings
$\str{A}\to f(\str{B})$ have the same colour.

The study of Ramsey classes emerged as a  generalisation of the
classical results of Ramsey theory.
Ramsey theorem itself implies that the class of all finite linearly ordered sets is Ramsey.
The first true structural Ramsey theorem was given in 1977 by Ne\v set\v ril and R\"odl~\cite{Nevsetvril1977} and, independently,
by Abramson and Harrington~\cite{Abramson1978}. By this result, the class of all finite graphs (or, more generally, relational structures in a given language) endowed with a linear ordering of vertices is Ramsey.
 The celebrated Ne\v set\v ril--R\"odl theorem~\cite{Nevsetvril1977} states that the classes of relational structures endowed with a linear order and omitting a given set of irreducible substructures is Ramsey. Additional examples are discussed in Section~\ref{sec:examples}.

\medskip

\paragraph{2.}
Class $\K$ of finite structures has the \emph{extension
property for partial automorphisms} (\emph{EPPA} for short, sometimes also called the
\emph{Hrushovski property}) if for every $\str{A}\in \K$ there exists
$\str{B}\in \K$ containing $\str{A}$ as an (induced) substructure 
such that every partial automorphism of $\str{A}$ (an
isomorphism between two induced substructures of $\str{A}$) extends to an automorphism of
$\str{B}$. 

This concept was originally motivated by the
\emph{small index conjecture} from group theory which states that a subgroup of the automorphism group of the countable random graph is of countable index if and only if it contains a stabiliser
of a finite set. This conjecture was proved by Hodges, Hodkinson, Lascar and Shelah in 1993~\cite{hodges1993b}.
To complete the argument
they used a result of Hrushovski~\cite{hrushovski1992} who, in 1992, established that the class of all finite graphs has EPPA.  After this, the quest of identifying new classes of structures with EPPA continued. Herwig generalised Hrushovski's result to all finite relational structures~\cite{Herwig1995} (in a fixed finite language) and later to the classes of all finite relational structures omitting certain sets of irreducible substructures~\cite{herwig1998}. He also developed several combinatorial strategies to establish EPPA for a given class. This development culminated to the Herwig--Lascar theorem~\cite{herwig2000}, one of the deepest results in the area, which established a structural condition for a class to have EPPA.  Again many more examples of classes having EPPA are known today and are discussed in Section~\ref{sec:examples}.

\bigskip

While EPPA and the Ramsey property may not seem related at first glance, they
are both (under mild and natural assumptions) strengthenings of a
model-theoretical notion of the amalgamation property (Definition~\ref{def:amalgamation}). It was noticed by Ne\v set\v ril in late
1980's~\cite{Nevsetvril1989a,Nevsetril2005} that every Ramsey class with the joint
embedding property has the amalgamation property. The same holds for EPPA as well
(in fact, EPPA was introduced as a stronger form of amalgamation). 

By the classical \Fraisse{} theorem~\cite{Fraisse1953} (Theorem~\ref{thm:unrestrictedEPPA}),
every amalgamation class $\K$ (that is, an isomorphism-closed hereditary class of finite structures with the
amalgamation property, the joint embedding, see Definition~\ref{defn:amalg}) with countably many mutually
non-isomorphic structures has an up to isomorphism unique countable \emph{\Fraisse{} limit} $\str H$.
The structure
$\str{H}$ is \emph{homogeneous}: every
partial automorphism  of $\str{H}$ with finite domain extends to an automorphism of $\str{H}$.
The correspondence between amalgamation classes and homogeneous structures is one-to-one, every amalgamation
class with only countably many mutually non-isomorphic structures is precisely the class of all finite structures which embed to $\str{H}$ (the \emph{age} of $\str{H}$).

In a surprising development with major implications for the
subject, in the period 2005--2008 both the Ramsey property and EPPA
were shown to be intimately linked to topological dynamics.
 In 2005, Kechris, Pestov and
Todor{\v c}evi{\' c}~\cite{Kechris2005} showed that the automorphism group of
a homogeneous structure $\str{H}$ is extremely amenable if and only if
its age is Ramsey. In 2008, Kechris and Rosendal~\cite{Kechris2007} showed
that EPPA implies amenability. 
This gave a fresh motivation to seek deeper understanding and additional examples of Ramsey and EPPA classes.
We discuss some recent developments in this area with a focus on progress in
systematising the combinatorial techniques on which the applications of the new
theory ultimately are based.
For further connections see, for example, Nguyen Van Th{\'e}'s~\cite{NVT14}, Bodirsky's~\cite{Bodirsky2015} and Solecky's~\cite{Solecki2014} surveys (for the Ramsey context); Macpherson's survey~\cite{Macpherson2011} (for homogeneous structures) and Siniora's thesis~\cite{Siniora2} (for the EPPA context).

\section{Preliminaries}
\label{sec:preliminaries}
\subsection{$\GammaL$-structures}
\label{sec:structures}
We find it convenient to work with model-theoretical structures (see, for example, \cite{Hodges1993}) generalised in two ways.

In the first place, we equip the language with a permutation
group. In the classical approach this group is trivial.

In the second place, we allow the language to include functions with values in the power set of the structure. This allows for an abstract treatment of the notion of algebraic closure (Definition~\ref{def:algebraic}).

 This extended formalism was introduced by Hubička, Konečný and Nešetřil to greatly simplify the presentation of EPPA constructions and also to make it easy to speak about free amalgamation classes in languages containing non-unary functions~\cite{Hubicka2018EPPA} (see also~\cite{Hubicka2016} and~\cite{Evans3} which use a related formalism in the Ramsey context).

\medskip

Let $L=L_{\mathcal R}\cup L_{\mathcal F}$ be a \emph{language} with \emph{relational symbols} $\rel{}{}\in L_{\mathcal R}$ and \emph{function symbols} $F\in L_{\mathcal F}$ each having its \emph{arity} denoted by $\arity{}$ for relations and $\arityf{}$ for functions.

Let $\GammaL$ be a permutation group on $L$ which preserves types and arities of all symbols. We say that $\GammaL$ is a \emph{language equipped with a permutation group}. 

Denote by $\powerset{A}$ the set of all subsets of $A$. A \emph{$\GammaL$-structure} $\str{A}$ is a structure with \emph{vertex set} $A$, functions $\func{A}{}\colon A^{\arityf{}}\to \powerset{A}$ for every $\func{}{}\in L_{\mathcal F}$ and relations $\rel{A}{}\subseteq A^{\arity{}}$ for every $\rel{}{}\in L_{\mathcal R}$.
(By $\powerset{A}$ we denote the power set of $\str{A}$.)
Notice that the domain of a function is ordered while the range is unordered.

\medskip

If the permutation group $\GammaL$ is trivial, we also speak of $L$-structures instead of $\GammaL$-structures.
The permutation of language is motivated by applications in the context of EPPA.
Presently there are no applications of this concept in the Ramsey context.
For this reason we formulate many definitions and results for $L$-structures, only so that we can refer to published proofs in the
area without the need to justify their validity for $\GammaL$-structures.

Note that we will still consider set-valued functions for $L$-structures.
If, in addition, all value ranges of all functions consist of singletons, one obtains the usual notion of model-theoretical language and structures. 
All results and constructions in this thesis presented on $\GammaL$-structures and $L$-structures can thus be applied in this context.

 If the vertex set $A$ is finite, we call $\str A$ a \emph{finite structure} and by $\Str(\GammaL)$ we denote the class of all finite $\GammaL$-structures.
We consider only structures with finitely or countably infinitely many vertices. 
If $L_{\mathcal F} = \emptyset$, we call $\GammaL$ a \emph{relational language} and say that a $\GammaL$-structure is  a \emph{relational $\GammaL$-structure}.
A function $\func{}{}$ such that $\arityf{}=1$ is a \emph{unary function}.

\subsection{Maps between $\GammaL$-structures}
The standard notions of homomorphism, monomorphism and embedding extend to 
$\GammaL$-structures naturally as follows.

\medskip

A \emph{homomorphism} $f\colon \str{A}\to \str{B}$ is a pair $f=(f_L,f_A)$ where $f_L\in \GammaL$ and $f_A$ is a mapping $A\to B$
 such that: 
\begin{enumerate}
\item[(a)] for every relation symbol  $\rel{}{}\in L_{\mathcal R}$ of arity $a$ it holds:
$$(x_1,x_2,\ldots, x_{a})\in \rel{A}{}\hskip-0.2cm\implies\hskip-0.2cm (f_A(x_1),f_A(x_2),\ldots,f_A(x_{a}))\allowbreak\in \permrel{f_L}{B}{},$$
\item[(b)] for every function symbol $\func{}{}\in L_{\mathcal F}$ of arity $a$ it holds:
$$f_A(\func{A}{}(x_1,x_2,\allowbreak \ldots, x_{a}))\subseteq\permfunc{f_L}{B}{}(f_A(x_1),f_A(x_2),\ldots,\allowbreak f_A(x_{a})).$$
\end{enumerate}
For brevity we will also write $f(x)$ for $f_A(x)$ in the context where $x\in A$ and $f(S)$ for $f_L(S)$ where $S\in L$.
For a subset $A'\subseteq A$ we denote by $f(A')$ the set $\{f(x): x\in A'\}$ and by $f(\str{A})$ the homomorphic image of a $\GammaL$-structure $\str{A}$.

If $f_A$ is injective then $f$ is called a \emph{monomorphism}. A monomorphism $f$ is an \emph{embedding} if:
\begin{enumerate}
\item[(a)] for every relation symbol  $\rel{}{}\in L_{\mathcal R}$ of arity $a$ it holds:
$$(x_1,x_2,\ldots, x_{a})\in \rel{A}{}\hskip-0.2cm\iff\hskip-0.2cm (f_A(x_1),f_A(x_2),\ldots,f_A(x_{a}))\allowbreak\in \permrel{f_L}{B}{},$$
\item[(b)] for every function symbol $\func{}{}\in L_{\mathcal F}$ of arity $a$ it holds:
$$f_A(\func{A}{}(x_1,x_2,\allowbreak \ldots, x_{a}))=\permfunc{f_L}{B}{}(f_A(x_1),f_A(x_2),\ldots,\allowbreak f_A(x_{a})).$$
\end{enumerate}
If $f$ is an embedding where $f_A$ is one-to-one then $f$ is an \emph{isomorphism}. If $f_A$ is an inclusion and $f_L$ is the identity then $\str{A}$ is a \emph{substructure} of $\str{B}$ and $f(\str{A})$ is also called a \emph{copy} of $\str{A}$ in $\str{B}$. For an embedding $f\colon\str{A}\to \str{B}$ we say that $\str{A}$ is \emph{isomorphic} to $f(\str{A})$.

Given $\str{A}\in \K$ and $B\subseteq A$, the \emph{closure of $B$ in $\str{A}$}, denoted by $\cl_{\str{A}}(B)$, is the smallest substructure of $\str{A}$ containing $B$.
For $x\in A$ we will also write $\cl_{\str{A}}(x)$ for $\cl_{\str{A}}(\{x\})$.

\subsection{Amalgamation classes, homogeneous structures and the \Fraisse{} theorem}\label{sec:amalg}
\begin{figure}
\centering
\includegraphics{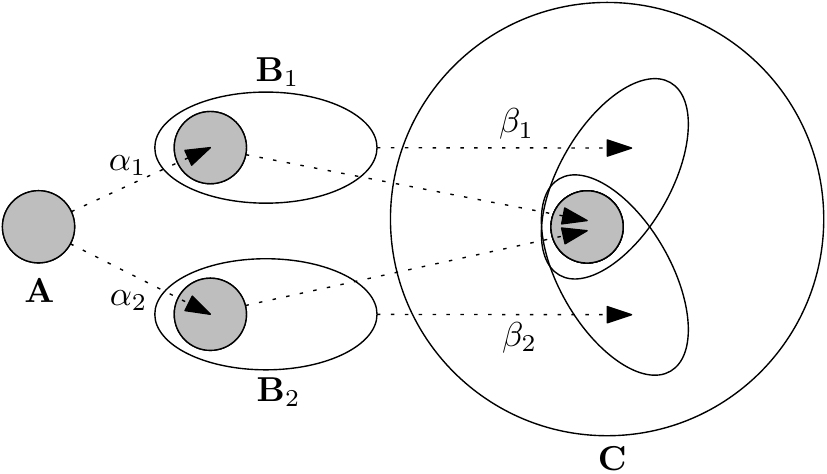}
\caption{An amalgamation of $\str{B}_1$ and $\str{B}_2$ over $\str{A}$.}
\label{amalgamfig}
\end{figure}

Basically all classes of structures of interest in this work are the so-called amalgamation classes
which we introduce now.

\begin{definition}[Amalgamation~\cite{Fraisse1953}]
\label{def:amalgamation}
Let $\str{A}$, $\str{B}_1$ and $\str{B}_2$ be $\GammaL$-structures, $\alpha_1$ an embedding of $\str{A}$ 
into $\str{B}_1$ and $\alpha_2$ an embedding of $\str{A}$ into $\str{B}_2$. Then
every $\GammaL$-structure $\str{C}$
 with embeddings $\beta_1\colon\str{B}_1 \to \str{C}$ and
$\beta_2\colon\str{B}_2\to\str{C}$ such that $\beta_1\circ\alpha_1 =
\beta_2\circ\alpha_2$ (note that this must also hold for the language part of $\alpha_i$'s and $\beta_i$'s) is called an \emph{amalgamation} of $\str{B}_1$ and $\str{B}_2$ over $\str{A}$ with respect to $\alpha_1$ and $\alpha_2$. 
\end{definition}
See Figure~\ref{amalgamfig}.
We will often call $\str{C}$ simply an \emph{amalgamation} of $\str{B}_1$ and $\str{B}_2$ over $\str{A}$
(in most cases $\alpha_1$ and $\alpha_2$ can be chosen to be inclusion embeddings). If the structure $\str{A}$ is empty, we call $\str{C}$ the \emph{joint embedding} of $\str{B}_1$ and $\str{B}_2$.

We say that the amalgamation is \emph{strong} if it holds that $\beta_1(x_1)=\beta_2(x_2)$ if and only if $x_1\in \alpha_1(A)$ and $x_2\in \alpha_2(A)$.
Strong amalgamation is \emph{free} if $C=\beta_1(B_1)\cup \beta_2(B_2)$ and whenever a tuple $\bar{x}$ of vertices of $\str{C}$ contains vertices of both
$\beta_1(B_1\setminus \alpha_1(A))$ and $\beta_2(B_2\setminus \alpha_2(A))$, then $\bar{x}$ is in no relation of $\str C$,
and for every function $\func{}{}\in L$ with $\arityf{} = |\bar{x}|$ it holds that $\func{C}{}(\bar{x})=\emptyset$.

\begin{definition}[Amalgamation class~\cite{Fraisse1953}]
\label{defn:amalg}
An \emph{amalgamation class} is a class $\K$ of finite $\GammaL$-structures which is closed for isomorphisms and satisfies the following three conditions:
\begin{enumerate}
\item \emph{Hereditary property:} For every $\str{A}\in \K$ and every structure $\str{B}$ with an embedding $f\colon \str B \to \str{A}$ we have $\str{B}\in \K$;
\item \emph{Joint embedding property:} For every $\str{A}, \str{B}\in \K$ there exists $\str{C}\in \K$ with embeddings $f\colon \str A\to \str{C}$ and $g\colon \str B\to \str C$;
\item \emph{Amalgamation property:} 
For $\str{A},\str{B}_1,\str{B}_2\in \K$ and embeddings $\alpha_1\colon\str{A}\to\str{B}_1$ and $\alpha_2\colon\str{A}\to\str{B}_2$, there is $\str{C}\in \K$ which is an amalgamation of $\str{B}_1$ and $\str{B}_2$ over $\str{A}$ with respect to $\alpha_1$ and $\alpha_2$.
\end{enumerate}
If the $\str{C}$ in the amalgamation property can always be chosen as the free amalgamation, then $\K$ is a \emph{free amalgamation class}. 
Analogously if $\str{C}$ can be always chosen as a strong amalgamation, then $\K$ is \emph{strong amalgamation class}
\end{definition}
A class $\K$ of finite $\GammaL$ has \emph{joint embedding property} if for every $\str{B}_1,\str{B}_2\in \K$ there exists a joint embedding $\str{C}\in \K$. 

\medskip

Recall that a $\GammaL$-structure is \emph{(ultra)homogeneous} if it is countable (possibly infinite) and every isomorphism between finite substructures extends to an automorphism.
A $\GammaL$-structure $\str{A}$ is \emph{locally finite} if the
$\str{A}$-closure of every finite subset of $A$ is finite. We focus on locally finite structures only. In
this context, (locally finite) homogeneous structures are characterised by the properties
of their finite substructures.  We denote by
$\Age(\str{A})$ the class of all finite structures which embed to $\str{A}$.
For a class $\K$ of relational structures, we denote by $\Age(\K)$ the class
$\bigcup_{\str{A}\in \K} \Age(\str{A})$.

The following is one of the cornerstones of model theory.

\begin{theorem}[\Fraisse{}~\cite{Fraisse1986}, see \eg{}\ \cite{Hodges1993}]
\label{fraissethm}
Let $\GammaL$ be language and let $\K$ be a class of finite $\GammaL$-structures with only countably many non-isomorphic structures.

\begin{enumerate}[label=$(\alph*)$]
\item\label{fraisse:a} A class $\K$ is the age of a countable
locally finite homogeneous structure $\str{H}$ if and only if $\K$ is an amalgamation
class.
\item If the conditions of \ref{fraisse:a} are satisfied then the structure $\str{H}$ is
unique up to isomorphism. 
\end{enumerate}
\end{theorem}
\medskip

Generalising notions of a graph clique and a Gaifman graph clique, we use the following:
\begin{definition}[Irreducible structure~\cite{Evans3}]
\label{def:irreducible}
A $\GammaL$-structure is \emph{irreducible} if it is not the free amalgamation of its proper
substructures.
\end{definition}\label{def:homomorphism-embedding}
The following captures the right notion of mapping which will be used to present
general constructions in Section~\ref{sec:general}.
\begin{definition}[Homomorphism-embedding~\cite{Hubicka2016}]

Let $\str{A}$ and $\str{B}$ be $\GammaL$-structures.
 A homomorphism $f\colon\str{A}\to \str{B}$ is
a \emph{homomorphism-embedding} if the restriction $f|_{\str C}$ is an embedding whenever $\str C$ is an irreducible
substructure of $\str{A}$.
\end{definition}
 Given a family $\mathcal F$ of $\GammaL$-structures, we denote by
$\Forb(\mathcal F)$ the class of all finite or finite $\GammaL$-structures $\str{A}$.  
such that there is no $\str{F}\in \mathcal F$ with a homomorphism-embedding $\str{F}\to\str{A}$.

\section{Extension property for partial automorphisms (EPPA)}
A \emph{partial automorphism} of a $\GammaL$-structure $\str{A}$ is an isomorphism $f\colon
\str{C} \to \str{C}'$ where $\str{C}$ and $\str{C}'$ are substructures of
$\str{A}$ (note that this also includes a full permutation $f_L\in \GammaL$ of the language).

\begin{definition}[EPPA~\cite{hrushovski1992}]
\label{defn:EPPA}
Let $\GammaL$ be a language equipped with a permutation group.
  We say that a class $\K$ of finite $\GammaL$-structures  has the {\em
extension property for partial automorphisms}  (\emph{EPPA}, sometimes called
the \emph{Hrushovski property}) if for every $\str{A} \in \K$
there is $\str{B} \in \K$ such that $\str{A}$ is a substructure of $\str{B}$
and every partial automorphism of $\str{A}$ extends to an automorphism of
$\str{B}$.
\end{definition}
  We call $\str{B}$ satisfying the condition of Definition~\ref{defn:EPPA} an \emph{EPPA-witness of
$\str{A}$}. $\str{B}$ is \emph{irreducible structure faithful} (with respect to $\str A$) if it has the property
that for every irreducible substructure $\str{C}$ of $\str{B}$ there exists an 
automorphism $g$ of $\str{B}$ such that $g(C)\subseteq A$.
The notion of faithful EPPA-witnesses was introduced by Hodkinson and Otto~\cite{hodkinson2003}.

\begin{observation}\label{obs:eppaamalgamation}
Every hereditary isomorphism-closed class of finite $\GammaL$-structures which has EPPA and the joint embedding property (see Definition~\ref{defn:amalg}) is an amalgamation class. 
\end{observation}
\begin{proof}
Let $\K$ be such a class and let $\str{A},\str{B}_1,\str{B}_2\in \K$, $\alpha_1\colon \str A \to \str B_1$, $\alpha_2\colon \str A\to \str B_2$ be as in Definition~\ref{defn:amalg}. Let $\str B$ be the joint embedding of $\str B_1$ and $\str B_2$ (that is, we have embeddings $\beta_1'\colon \str B_1\to \str B$ and $\beta_2'\colon \str B_2\to \str B$) and let $\str C$ be an EPPA-witness for $\str B$.

Let $\varphi$ be a partial automorphism of $\str B$ sending $\alpha_1(\str A)$ to $\alpha_2(\str A)$ and let $\theta$ be its extension to an automorphism of $\str C$. Finally, put $\beta_1 = \theta\circ\beta_1'$ and $\beta_2 = \beta_2'$. It is easy to check that $\beta_1$ and $\beta_2$ certify that $\str C$ is an amalgamation of $\str B_1$ and $\str B_2$ over $\str A$ with respect to $\alpha_1$ and $\alpha_2$.
\end{proof}
EPPA was introduced by Hodges, Hodkinson, Lascar, and Shelah~\cite{hodges1993b} to show the existence of ample generics (see \eg~\cite{Siniora2}). For this application an additional property of the class is needed.
\begin{definition}[APA~\cite{hodges1993b}]\label{defn:apa}
Let $\K$ be a class of finite $\GammaL$-structures. We say that $\K$ has the \emph{amalgamation property with automorphisms} (\emph{APA}) if for every $\str A, \str B_1, \str B_2 \in \mathcal C$ such that $\str A\subseteq \str B_1,\str B_2$ there exists $\str C\in \K$ which is an amalgamation of $\str B_1$ and $\str B_2$ over $\str A$, has $\str B_1,\str B_2\subseteq \str C$ and furthermore the following holds:

For every pair of automorphisms $f_1\in\Aut(\str B_1)$, $f_2\in \Aut(\str B_2)$ such that $f_i(A) = A$ for $i\in\{1,2\}$ and $f_1|_A = f_2|_A$, there is an automorphism $g\in \Aut(\str C)$ such that $g|_{B_i} = f_i$ for $i\in\{1,2\}$.
\end{definition}

Siniora and Solecki~\cite{solecki2009,Siniora} strengthened the notion of EPPA in order to get a dense locally finite subgroup of the automorphism group of the corresponding \Fraisse{} limit.

\begin{definition}[Coherent maps~\cite{Siniora}]
Let $X$ be a set and $\mathcal P$ be a family of partial bijections 
 between subsets
of $X$. A triple $(f, g, h)$ from $\mathcal P$ is called a \emph{coherent triple} if $$\dom(f) = \dom(h), \range(f ) = \dom(g), \range(g) = \range(h)$$ and $$h = g \circ f.$$

Let $X$ and $Y$ be sets, and $\mathcal P$ and $\mathcal Q$ be families of partial bijections between subsets
of $X$ and between subsets of $Y$, respectively. A function $\varphi\colon \mathcal P \to \mathcal Q$ is said to be a
\emph{coherent map} if for each coherent triple $(f, g, h)$ from $\mathcal P$, its image ($\varphi(f), \varphi(g), \varphi(h)$) in $\mathcal Q$ is also coherent.
\end{definition}
\begin{definition}[Coherent EPPA~\cite{Siniora}]
\label{defn:coherent}
A class $\K$ of finite $\GammaL$-structures is said to have \emph{coherent EPPA} if $\K$ has EPPA and moreover the extension of partial automorphisms
is coherent. That is, for every $\str{A} \in \K$, there exists $\str{B} \in \K$ such that $\str A\subseteq \str B$ and every
partial automorphism $f$ of $\str{A}$ extends to some $\hat{f} \in \Aut(\str{B})$ with  the property that the map $\varphi$ from the partial automorphisms of $\str{A}$ to $\Aut(\str{B})$ given by $\varphi(f) = \hat{f}$ is coherent. We also say that $\str B$ is a \emph{coherent EPPA-witness} for $\str A$.
\end{definition}

\section{Ramsey classes}
 For $L$-structures $\str{A},\str{B}$, we denote by ${\str{B}\choose \str{A}}$ the set of all embeddings $\str{A}\to\str{B}$. Using this notation, the definition of a Ramsey class gets the following form:

\begin{definition}[Ramsey class]
\label{def:Ramsey}
Let $L$ be a language.  A class $\K$ of finite $L$-structures is a \emph{Ramsey class} (or has the \emph{Ramsey property}) if for every two $L$-structures $\str{A}\in \K$ and $\str{B}\in\K$ and for every positive integer $k$ there exists a $L$-structure $\str{C}\in\K$ such that the following holds: For every partition of ${\str{C}\choose \str{A}}$ into $k$ classes there is $\widetilde{\str B} \in {\str{C}\choose \str{B}}$ such that ${\widetilde{\str{B}}\choose \str{A}}$ belongs to one class of the partition.  It is usual to shorten the last part of the definition to $\str{C} \longrightarrow (\str{B})^{\str{A}}_k$.
\end{definition}
Analogously to Observation~\ref{obs:eppaamalgamation} one can see the following:
\begin{observation}[Ne\v set\v ril, 1980's~\cite{Nevsetvril1989a}]\label{obs:ramseyamalgamation}
Every hereditary iso\-morphism-closed Ramsey class of finite $L$-structures with the joint embedding property is an amalgamation class. 
\end{observation}
\begin{proof}
Let $\K$ be such a class and let $\str{A},\str{B}_1,\str{B}_2\in \K$, $\alpha_1\colon \str A \to \str B_1$, $\alpha_2\colon \str A\to \str B_2$ be as in Definition~\ref{defn:amalg}. Let $\str B$ be the joint embedding of $\str B_1$ and $\str B_2$ and let $\str C\to (\str{B})^\str{A}_2$. 
Colour the elements of $\binom{\str{C}}{\str{A}}$ according to whether or not they are contained in a copy of $\str{B}_1$ in $\str{C}$.
 Because $\str{C}$ is Ramsey there is a monochromatic copy $\widetilde{\str{B}}\in {\str{C}\choose \str{B}}$.
 As it contains a copy of $\str{B}_1$, all the copies of $\str{A}$ in it are in a copy of $\str{B}_1$ in $\str{C}$. But this includes the $\str{A}$ which is in the copy of $\str{B}_2$ in $\widetilde{\str B}$.
It follows that $\str{C}$ contains an amalgamation of $\str{B}_1$ and $\str{B}_2$ over $\str{A}$.
\end{proof}
\section{Nešetřil's classification programme}
\label{sec:classification}

The celebrated (Cherlin and Lachlan's) \emph{classification programme of homogeneous structures}
aims to provide full catalogues of countable
homogeneous combinatorial structures of a given type. The most important cases where the classification
is complete are:
\begin{enumerate}
\item\label{cat1} Schmerl's catalogue of all homogeneous partial orders~\cite{Schmerl1979},
\item Lachlan and Woodrow's catalogue of all homogeneous simple graphs~\cite{Lachlan1980},
\item\label{cat3} Lachlan's catalogue of all homogeneous tournaments~\cite{lachlan1984countable},
\item Cherlin's catalogue of all homogeneous digraphs~\cite{Cherlin1998} (this generalises catalogues \ref{cat1} and \ref{cat3}),
\item Cherlin's catalogue of all ordered graphs~\cite{Cherlin2013}, and
\item Braunfeld's catalogue of homogeneous finite-dimensional permutation structures~\cite{SamPhD}. This was recently shown to be complete by Braunfeld and Simon~\cite{sam2,braunfeld2018classification}.
\end{enumerate}
Several additional catalogues are known. An extensive list is given in the upcoming Cherlin's monograph~\cite{Cherlin2013}.
Important for our discussion is also Cherlin's catalogue of metrically homogeneous graphs which is currently conjectured to be complete~\cite{Cherlin2013}.

All these results give us a rich and systematic source of homogeneous structures.

\subsection{Classification of Ramsey expansions}

Shortly after observing the link between amalgamation classes, homogeneous structures and Ramsey
classes, Ne\v set\v ril \cite{Nevsetvril1989a} studied the catalogue of homogeneous graphs and verified that the corresponding Ramsey results are already known. In 2005, he proposed a project to
analyse these catalogues and initiated the classification programme of Ramsey
classes (which we refer to as \emph{Ne\v set\v ril's classification programme})~\cite{Nevsetril2005,Hubicka2005a}. 
The overall idea is symbolised in \cite{Nevsetril2005} as follows:
\begin{center}
\begin{tikzpicture}[auto,
    box/.style ={rectangle, draw=black, thick, fill=white,
      text width=9em, text centered,
      minimum height=2em}]
     \tikzstyle{line} = [draw, thick, -latex',shorten >=2pt];
    \matrix [column sep=5mm,row sep=3mm] {
      \node [box] (Ramsey) {Ramsey\\ classes};
      &&\node [box] (amalg) {amalgamation classes};
      \\
\\
      \node [box] (lift) {Ramsey structures};
      &&\node [box] (lim) {(locally finite)\\ homogeneous structures};
      \\
    };
    \begin{scope}[every path/.style=line]
      \path (Ramsey)   -- (amalg);
      \path (amalg)   -- (lim);
      \path (lim)   -- (lift);
      \path (lift)   -- (Ramsey);
    \end{scope}
  \end{tikzpicture}
\end{center}
The individual arrows in the diagram can be understood as follows.
\begin{enumerate}
\item \emph{Ramsey classes $\longrightarrow$ amalgamation classes}.  By Observation~\ref{obs:ramseyamalgamation}, every hereditary iso\-morphism-closed Ramsey class of finite $L$-structures with the joint embedding property is an amalgamation class.

The assumptions of Observation~\ref{obs:ramseyamalgamation} are natural in the Ramsey context and it makes sense to restrict our attention only to the classes which satisfy them.
\item \emph{Amalgamation classes $\longrightarrow$ homogeneous structures}. By the \Fraisse{} theorem (Theorem~\ref{thm:unrestrictedEPPA}), every amalgamation class of $L$-structures with countably many mutually non-isomorphic structures has a \Fraisse{}~limit which is a (locally finite) homogeneous structure. 

The additional assumption about the amalgamation class having only countably many mutually non-isomorphic structures is a relatively mild one. However, there are interesting and natural Ramsey classes not satisfying it, such as the class of all finite ordered metric spaces (see Section~\ref{sec:metric}).
\item \emph{Ramsey structures $\longrightarrow$ Ramsey classes}. We call a structure \emph{Ramsey} if it is locally finite and its age is a Ramsey class and thus this connection follows by the definition.

\end{enumerate}
The difficult part of the diagram is thus the last arrow \emph{homogeneous structures $\longrightarrow$ Ramsey structures}. This motivates a question:
\begin{quote}
Is every locally finite homogeneous structure also a Ramsey structure?
\end{quote}
The answer is most definitely negative, for reasons to be discussed below. For this reason, the precise formulation of this
classification programme requires close attention.

It is a known fact that the automorphism group of every Ramsey structure fixes a linear order on
vertices. This follows at once from the connection with topological dynamics~\cite{Kechris2005}; see \cite[Proposition 2.22]{Bodirsky2015} for a combinatorial proof.
This is a very strong hypothesis---such a linear order is unlikely to appear in practice unless we began with a class of
ordered structures. So it often occurs that the 
age $\K$ of a homogeneous structure $\str{H}$ is not Ramsey.
In many such cases, it can be ``expanded'' to a Ramsey class $\K^+$ by adding an additional
binary relation $\leq$ to the language and considering every structure in $\K$ with
every possible order.  It is easy to check that $\K^+$ is an amalgamation class and
the \Fraisse{} limit of $\K^+$ thus exists and leads to a Ramsey structure $\str{H}^+$.
This Ramsey structure can be thought of as $\str{H}^+$ with an additional ``free'' or ``generic'' linear
order of vertices.

\begin{example}[Countable random graph]
Consider the class $\mathcal G$ of all finite graphs and
its~\Fraisse{} limit $\str{R}$. ($\str{R}$ is known as the \emph{countable random
graph}, or the \emph{Rado graph}, and is one of the structures in the Lachlan and Woodrow's catalogue.)  Because an order is not fixed by $\Aut(\str{R})$, we can
consider the class $\mathcal G^+$ of all finite graphs ``endowed'' with a linear order of vertices.
More formally, $\mathcal G^+$ is the class of all finite $L$-structures $\str{A}$, where $L$ consists of two binary relational symbols $E$ (for edges) and $\leq$ (for the order), $E_\str{A}$ is a symmetric
and irreflexive relation on $A$ (representing the edges of a graph on $A$) and $\leq_\str{A}$ is a linear order of $A$.
By the Ne\v set\v ril--R\"odl theorem (Theorem \ref{thm:NRoriginal}), $\mathcal G^+$ is Ramsey
and thus its~\Fraisse{} limit $\str{R}^+$ is a Ramsey structure.
In this sense, we completed the last arrow of Ne\v set\v ril's diagram.
\end{example}

Initially, the classification problem was understood in terms
of adding linear orders freely when they were absent, and
then confining one’s attention to classes with the linear order
present. This point of view is still implicit in \cite{Kechris2005}.

However, it turns out that it is necessary to consider more general expansions
of languages than can be afforded using a single linear order. Furthermore, the
topological dynamical view of~\cite{Kechris2005} clarifies what kind of Ramsey
expansion we actually should look for, and this can be expressed very
concretely in combinatorial terms.

On the other hand, if one allows more general expansions of the structure---for
example, naming every point---then the Ramsey property may be obtained in a
vacuous manner.  This means that a conceptual issue of identifying the
``correct'' Ramsey expansions needs to be resolved first before returning to
the technical issues of proving the existence of such an expansion, and
constructing it explicitly.
Nowadays, the conceptual issue may be considered to be satisfactorily resolved (at least
provisionally) and progress on the resulting combinatorial problems is one of
the main subjects of this habilitation.

We review the resolution of this issue and justify the following:
the last arrow in Ne\v set\v ril’s
diagram represents the search for an \emph{optimal expansion} of the
homogeneous structure to a larger language, in a precise sense,
and raises the question of the existence of such expansion.

Before entering into the technicalities associated with the conceptual issues,
we present a critical example in which the Ramsey property requires something
more than the addition of a linear order, and we describe the optimal solution
according to the modern point of view.

\begin{example}[Generic local order]
\label{example:S2}
An early example (given by Laflamme, Nguyen Van Th\'e and Sauer~\cite{laflamme2010partition}) of a structure
with a non-trivial optimal Ramsey expansion is the \emph{generic local order}.
This is a homogeneous tournament $\str{S}(2)$ (which we view as an $L$-structure where $L$ consists of binary relational symbol $E$) defined as follows.   Let $\mathbb
T$ denote the unit circle in the complex plane.  Define an oriented graph
structure on $\mathbb T$ by declaring that there is an arc from $x$ to $y$ iff
$0<\mathop{\mathrm{arg}}\nolimits(y/x)< \pi$.  Call $\vv{\mathbb T}$ the resulting oriented graph.  The
dense local order is then the substructure $\str{S}(2)$ of $\vv{\mathbb T}$ whose
vertices are those points of $\mathbb T$ with rational argument.

Another construction of $\str{S}(2)$ is to start from the order of rationals
(seen as a countable transitive tournament), randomly colour vertices with two
colours and then reverse the direction of all arrows between vertices of different
colours. In fact, this colouring is precisely the necessary Ramsey expansion.
We thus consider the class of finite $L^+$-structures where 
$L^+$ consists of a binary relation $E$ (representing the directed edges) and a
unary relation $R$ (representing one of the colour classes).
The linear ordering of vertices is implicit, but can be defined based
on the relations $E$ and $R$, putting $a\leq b$ if and only if either there
is an edge from $a$ to $b$ and they belong to the same colour class, or there
is an edge from $b$ to $a$ and they belong to different colour classes.
The Ramsey property then follows by a relatively easy application of the Ramsey theorem.
\end{example}
The general notion of expansion (or, more particularly, \emph{relational expansion}) that we work with comes from model theory.

\begin{definition}[Expansion and reduct]
Let $L^+=L^+_\mathcal R\cup L^+_\mathcal F$ be a language containing  language $L=L_\mathcal R\cup L_\mathcal F$, extending it by relational symbols. By this we mean $L_\mathcal R\subseteq L_\mathcal R^+$ and $L_\mathcal F = L^+_\mathcal F$ and the arities of the relations and functions which belong to both $L$ and $L^+$ are the same.
For every structure $\str{X}\in \Str(L^+)$, there is a unique structure $\str{A}\in \Str(L)$ satisfying $A=X$, $\rel{A}{}=\rel{X}{}$ for every $\rel{}{}\in L_\mathcal R$ and $\func{A}{}=\func{X}{}$ for every $\func{}{}\in L_\mathcal F$.
 We call $\str{X}$ a \emph{(relational) expansion} (or a \emph{lift}) of $\str{A}$ and $\str{A}$ is called the \emph{$L$-reduct} (or the \emph{$L$-shadow}) of $\str{X}$. 

 Given languages $L$ and $L^+$ and a class $\K$ of $L$-structures, a class $\K^+$  of $L^+$-structures is an \emph{expansion} of $\K$ if $\K$ is precisely the class of all $L$-reducts of structures in $\K^+$.
\end{definition}
As mentioned earlier, every $L$-structure $\str{H}$ can be turned to a Ramsey
structure by naming every point. This can be done by expanding the language $L$
by infinitely many unary relational symbols and putting every vertex of
$\str{H}$ to a unique relation. The Ramsey property then holds,
since there are no non-trivial embeddings between structures from the age of $\str H$. Clearly, additional
restrictions on the expansions need to be made. 

We now formulate a notion of ``canonical'' or ``minimal''
Ramsey expansion. We will give this first in purely combinatorial terms. In those terms, we seek a ``precompact
Ramsey expansion with the expansion property'' as defined
below. But to understand why this is canonical (\ie{}, natural and unique up to bi-definability), we need to invoke notions and 
non-trivial results of topological dynamics.

\begin{definition}[Precompact expansion~\cite{The2013}]
\label{defn:precompact}
Let $\mathcal K^+$ be an expansion of a class of structures $\K$.
We say that $\mathcal K^+$ is a \emph{precompact expansion} of $\mathcal K$ if for
every structure $\str{A} \in \mathcal K$ there are only finitely many
structures $\str{A}^+ \in \mathcal K^+$ such that $\str{A}^+$ is an expansion of
$\str{A}$.
\end{definition}
\begin{definition}[Expansion property~\cite{The2013}]
\label{defn:ordering}
Let $\mathcal K^+$ be an expansion of $\K$. For $\str{A},\str{B}\in \K$ we say
that $\str{B}$ has the \emph{expansion property} for $\str{A}$ if for every expansion $\str{B}^+\in \mathcal K^+$ of $\str{B}$ and for every expansion $\str{A}^+\in \mathcal K^+$ of $\str{A}$ there is an embedding $\str A^+\to\str{B}^+$.

$\mathcal K^+$ has the \emph{expansion property} relative to $\K$ if for every $\str{A}\in \K$
there is $\str{B}\in \K$ with the expansion property for $\str{A}$.
\end{definition}

Intuitively, precompactness means that the expansion is not very rich
and the expansion property then shows that it is minimal possible. 
To further motivate these concepts, we review the key connections to topological
dynamics. 

We consider the automorphisms group $\Aut(\str{H})$ as a Polish topological
group by giving it the topology of pointwise convergence.  Recall  that a
topological group $\Gamma$ is \emph{extremely amenable} if whenever $X$ is a
\emph{$\Gamma$-flow} (that is, a non-empty compact Hausdorff $\Gamma$-space on
which $\Gamma$ acts continuously), then there is a $\Gamma$-fixed point in $X$. See \cite{NVT14} for details.

In 1998, Pestov~\cite{Pestov1998free} used the classical Ramsey theorem to show
that the automorphism group of the order of rationals is extremely amenable. Two years later, Glasner and
Weiss~\cite{glasner2002minimal} proved (again applying the Ramsey theorem) that the
space of all linear orderings on a countable set is the universal minimal flow
of the infinite permutation group.  In 2005, Kechris, Pestov and Todor\v cevi\' c
introduced the general framework (which we refer to as \emph{KPT-correspondence}) connecting \Fraisse{} theory, Ramsey classes,
extremely amenable groups and metrizable minimal flows.
Subsequently, this framework was generalised to the notion of Ramsey expansions~\cite{The2013,NVT2009,Melleray2015,zucker2016topological} with main results as follows:

\begin{theorem}[{Kechris, Pestov, Todor\v cevi\' c \cite[Theorem 4.8]{Kechris2005}}]
\label{thm:KPT}
 Let $\str{H}$ be locally finite homogeneous $L$-structure. Then $\Aut(\str{H})$ is extremely amenable if and only if $\Age(\str{H})$ is a Ramsey class.
\end{theorem}
Theorem~\ref{thm:KPT} is often formulated with the additional assumption that $\Age(\str{H})$ is rigid (i.e. no structure in $\Age(\str{H})$ has non-trivial automorphisms).  This is however implied by our
definition of a Ramsey class (Definition~\ref{def:Ramsey}) which colours embeddings. This definition implies rigidity.
As mentioned earlier, in addition to having a rigid age, the automorphism group of $\str{H}$ must also fix a linear order.

Recall that a $\Gamma$-flow is \emph{minimal} if it admits no nontrivial closed $\Gamma$-invariant subset or, equivalently, if every orbit is dense. 
Among all minimal $\Gamma$-flows, there exists a canonical one, known as the \emph{universal minimal $\Gamma$-flow}.
Precompact Ramsey expansions with the expansion property relate to universal minimal flows as follows.
\begin{theorem}[{Melleray, Nguyen Van Th\'e, and Tsankov \cite[Corollary 1.3]{Melleray2015}}]
\label{thm:metrizable}
Let $\str{H}$ be a locally finite homogeneous structure and let $\Gamma=Aut(\str{H})$. The following are equivalent:
\begin{enumerate}
\item The universal minimal flow of $\Gamma$ is metrizable and has a comeagre orbit.
\item The structure $\str{H}$ admits a precompact expansion $\str{H}^+$ whose age has the Ramsey property, and has the expansion property relative to $\Age(\str{H})$.
\end{enumerate}
\end{theorem}
Because the metrizable minimal flow is unique, a corollary of
Theorem~\ref{thm:metrizable} is that for a given homogeneous structure $\str{H}$
there is, up to bi-definability, at most one Ramsey expansion $\str{H}^+$ such that
$\Age(\str{H}^+)$ is a precompact Ramsey expansion of $\Age(\str{H})$ with the
expansion property (relative to $\str{H})$. We will thus call such an expansion
the \emph{canonical Ramsey expansion}.

\medskip

The classification programme of Ramsey classes thus turns into two questions.
Given a locally finite homogeneous $L$-structure $\str{H}$, we ask the following:
\begin{enumerate}[label=Q\arabic*]
 \item \label{Q1} Is there a Ramsey structure $\str{H}^+$ which is a (relational) expansion of $\str{H}$
  such that $\Age(\str{H})^+$ is a precompact expansion of $\Age(\str{H})$? (Possibly with $\str{H}^+=\str{H}$.)
\end{enumerate}
If the answer to \ref{Q1} is positive, we know that $\str{H}^+$ can be
chosen so that $\Age(\str{H})^+$ has the expansion property relative to
$\Age(\str{H})$~\cite[Theorem 10.7]{Kechris2005}. We can moreover ask.
\begin{enumerate}[label=Q\arabic*,resume]
 \item \label{Q2}
If the answer to \ref{Q1} is positive, can we give an explicit description of $\str{H}^+$ which
additionally satisfies that the $\Age(\str{H}^+)$ has the expansion property with respect to $\Age(\str{H})$?
In other words, can we describe the canonical Ramsey expansion of $\str{H}$?
\end{enumerate}
In 2013, the classification programme in this form was first completed by
Jasi{\'n}ski, Laflamme, Nguyen Van Th{\'e} and Woodrow~\cite{Jasinski2013}
for the catalogue of homogeneous digraphs. More examples are discussed below.

\medskip

\begin{remark}
In addition to the universal minimal flow (Theorem~\ref{thm:metrizable}),
by a counting argument given by Angel, Kechris and Lyons~\cite{AKL14},
knowledge of an answer to question \ref{Q2} often gives amenability of
$\Aut(\str{H})$ and under somewhat stronger assumptions also shows that
$\Aut(\str{H})$ is uniquely ergodic. See
\cite{sokic2015semilattices,PawliukSokic16,jahel2019unique} for an initial
progress on the classification programme in this direction.
\end{remark}
\subsection{Classification of EPPA classes}

In the light of recent connections between EPPA and Ramsey classes, we can extend
this programme to also provide catalogues of classes with EPPA. The basic
question for a given a locally finite homogeneous $\GammaL$-structure $\str{H}$ is simply:

\begin{enumerate}[label=Q\arabic*,resume]
 \item \label{Q3} Does the class $\Age(\str{H})$ have EPPA?
\end{enumerate}
This is an interesting question from the combinatorial point of view
and again has number of applications. Motivated by a group-theoretical context one
can additionally consider the following questions:
\begin{enumerate}[label=Q\arabic*,resume]
 \item \label{Q4} If the answer to \ref{Q3} is positive, does it have coherent EPPA (Definition~\ref{defn:coherent})?
 \item \label{Q5} If the answer to \ref{Q3} is positive, does it have APA (Definition~\ref{defn:apa})?
\end{enumerate}
In this sense, the classification was first considered by Aranda, Bradley-Williams, Hubi{\v c}ka, Karamanlis, Kompatscher, Kone{\v c}n{\'y} and Pawliuk for metrically homogeneous graphs~\cite{Aranda2017}.
\begin{remark}[On group-theoretical context]

By a result of Kechris and Rosendal \cite{Kechris2007}, a positive answer to \ref{Q3} implies amenability of $\Aut(\str{H})$.  This is a sufficient
but not a necessary condition.  For example, the automorphism group of the order of rationals is
amenable (because it is extremely amenable) but does not have EPPA.

A positive answer to \ref{Q4} implies the existence of a dense locally finite
subgroup~\cite{solecki2009,Siniora}. The first known example, where an answer to
question~\ref{Q4} seems to be non-trivial (and presently open) is the class of
two-graphs where EPPA was shown by Evans, Hubi\v cka, Kone\v cn\'y and Ne\v set\v ril~\cite{eppatwographs}.
The existence of a dense locally finite subgroup here however follows by a different argument.
\end{remark}

\medskip
\subsection{Known classifications}

The Ramsey part of the classification programme was completed for the following
catalogues:
\begin{enumerate}
\item The catalogue of homogeneous graphs (Ne\v set\v ril~\cite{Nevsetvril1989a}).
We remark that by today's view of the programme, the original 1989 results are
incomplete because only expansions by free orderings are considered. The
remaining cases (of disjoint unions of cliques and their complements) are
however very simple.
\item The catalogue of directed graphs (Jasi{\'n}ski, Laflamme, Nguyen Van Th{\'e} and Woodrow~\cite{Jasinski2013}).
\item Completing the first two catalogues also covers all cases of the catalogue of homogeneous ordered graphs
with the exception of the class of all finite partial orders with a linear extension which is known to be Ramsey, too (Section~\ref{sec:partialorders}).
\item The conjectured-to-be-complete catalogue of metrically homogeneous graphs (Aranda, Bradley-Williams, Hubi{\v c}ka, Karamanlis, Kompatscher, Kone{\v c}n{\'y}, Pawliuk~\cite[Theorem 1.1]{Aranda2017}).
\end{enumerate}
Even though this point has apparently not been explicitly stated before, the EPPA classification
is finished for the catalogue of homogeneous graphs. In fact, this catalogue consists
of only 5 types of structures and their complements:
\begin{enumerate}
 \item The countable random graph $\str{R}$. $\Age(\str{R})$ is the class of all finite graphs for which
EPPA was shown by Hrushovski~\cite{hrushovski1992}.

Coherent EPPA was given by Siniora and Solecki~\cite{Siniora} and APA is trivial for every free amalgamation class.
 \item Generic $K_k$-free graphs $\str{R}_k$, $k\geq 3$. $\Age(\str{R}_k)$ is the
class of all $K_k$-free graphs. For $k=3$, EPPA was shown by Herwig~\cite{Herwig1995} and he
later generalised the construction to $k>3$~\cite{herwig1998}.

Again, the existence of coherent EPPA extensions (using a construction by
Hodkinson and Otto~\cite{hodkinson2003}) was verified by Siniora and
Solecki~\cite{Siniora} and APA is trivial.
 \item The class of all graphs consisting of at most $n$ cliques each of size at most $k$
for a given $n,k\in \mathbb N\cup \{\infty\}$, $n+k=\infty$.
Proving (coherent) EPPA for the ages of these structures is an easy exercise (building on the fact
that every partial permutation extends coherently to a permutation).

The class has APA if and only if $n,k\in\{1,\infty\}$. If $n\geq 2$ is finite,
one may consider an amalgamation of an $n$-anti-clique (that is, a graph with $n$
vertices and no edges) and a 2-anti-clique over the empty graph. Similarly, for $k\geq
2$ finite, one can consider an amalgamation of a $k$-clique and a $2$-clique over a vertex.
These amalgamations are counter-examples to APA.

In this situation, it is possible to modify the structures by 
adding vertices representing imaginaries (Section~\ref{sec:equivalences}) and 
functions representing the algebraic closures (Section~\ref{sec:closures}), including, possibly, constants
for the algebraic closure of the empty set. The age
of the resulting structure then has coherent EPPA and APA.
\end{enumerate}
It is also clear that if $\K$ is a class of graphs with EPPA, then the class of all complements
of graphs in $\K$ has EPPA, too.

EPPA for ages of homogeneous directed graphs was analysed by Pawliuk and Soki{\'c}~\cite{PawliukSokic16}. Their work
leaves several open cases: $n$-partite tournaments, semi-generic tournaments (for both these cases EPPA was claimed recently by Hubi\v{c}ka, Jahel, Kone{\v c}n{\'y}, and Sabok~\cite{HubickaSemigeneric}), tournaments (which present a well known open problem in the area asked in 2000 by Herwig and Lascar~\cite{herwig2000}), directed graphs omitting an independent set of size $n\geq 2$ and 2-covers of generic tournaments. 
The last two cases appear to be very similar to the case of tournaments.

Finally, Aranda, Bradley-Williams, Hubi{\v c}ka, Karamanlis, Kompa\-tscher, Kone{\v c}n{\'y} and Pawliuk studied EPPA for ages of metric spaces associated with metrically homogeneous graphs (and, for the first time, considered both the Ramsey property and EPPA together) and characterised all structures in the catalogue with the exception of special cases of antipodal metrically homogeneous graphs~\cite[Theorem 1.2]{Aranda2017}. 
For the Ramsey property (and partly also for EPPA), the analysis of metrically homogeneous graphs was, for the first time,
done using the general results which are introduced in Section~\ref{sec:general}. 
That work can also be seen as a confirmation of the effectivity of these methods.
EPPA for the remaining case (of antipodal metric spaces) was recently proved by Kone\v cn\'y~\cite{Konecny2019a} generalising a result
for diameter three by Evans, Hubi{\v c}ka, Kone{\v c}n{\'y} and Ne\v set\v ril~\cite{eppatwographs}.

Results on metrically homogeneous graphs are all based on a close study of \emph{completion algorithms} for partial structures. 
The completion algorithm introduced then led to a new line
of research on generalised metric spaces~\cite{Hubicka2017sauerconnant,Hubicka2017sauer,Konecny2018bc,Konecny2018b}. This in turn led to new tools
for proving simplicity of the automorphism groups of a wide range of structures
\cite{Evanssimplicity}, developing a method of Tent and Ziegler with its roots in model
theory (stability theory) \cite{Tent2013,Tent2013b,Li2018}. As a result, there is a new framework for
understanding various structural properties of these combinatorial structures
\cite{Hubickacycles2018} which may even lead to a more profound understanding of their
classification.

\section{General constructions}
\label{sec:general}
In this section, we present general theorems which are then used to show that a given
class of structures has EPPA or the Ramsey property. These results unify essentially
all earlier results in the area under a common framework as discussed in Section~\ref{sec:examples}
which also lists all known examples which were not yet covered by this work.

\subsection{Unrestricted theorems}
The starting points for subsequent constructions of EPPA-witnesses and Ramsey
structures are the following two theorems. We will refer to them as
\emph{unrestricted theorems}, because no restrictions on the constructed structures are given.

\begin{theorem}[{Hubi\v cka, Kone\v cn\'y, Ne\v set\v ril 2019~\cite[Theorem 1.3]{Hubicka2018EPPA}}]
\label{thm:unrestrictedEPPA}
Let $\GammaL$ be a finite language equipped with a permutation group where all function symbols are unary (no restrictions are given on the relational symbols) and $\str{A}$
be a finite $\GammaL$-structure. Then there exists a finite coherent EPPA-witness $\str{B}$ for $\str A$.

Consequently, the class $\Str(\GammaL)$ of all finite $\GammaL$-structures has coherent EPPA.
\end{theorem}
For a language $L$ containing a binary relation $\leq$, we say that an $L$-structure $\str{A}$
is \emph{ordered} if $\leq_\str{A}$ is a linear order of $A$.
\begin{theorem}[{Hubi\v cka, Ne\v set\v ril 2019~\cite[Theorem 2.19]{Hubicka2016}}]
\label{thm:unrestrictedRamsey}
Let $L$ be a language containing a binary relation $\leq$ and $\str{A}$, $\str{B}$
be finite ordered $L$-structures. Then there exists a finite ordered $L$-structure $\str{C}$
such that $\str{C}\longrightarrow (\str{B})^\str{A}_2$.

Consequently, the class $\OStr(L)$ of all finite ordered $L$-structures is Ramsey.
\end{theorem}
Recall that (under mild assumptions) every Ramsey class fixes a linear order (Section~\ref{sec:classification}).
For this reason, the assumption on structures being
ordered in Theorem~\ref{thm:unrestrictedRamsey} can not be omitted and thus Theorem~\ref{thm:unrestrictedRamsey}
is the most general unrestricted Ramsey theorem for finite $L$-structures.

Except for
degenerated examples, EPPA classes never consists of ordered structures: automorphism
groups of finite linearly ordered chains are trivial and thus can not extend any
non-trivial partial automorphism (such as one sending a vertex to another).
Classes with EPPA and Ramsey property are thus basically disjoint.
Yet the similarity between both theorems shows that both types of classes are related.

Despite the compact formulations, both theorems are a result of a long development.

\begin{remark}[History of unrestricted EPPA results]
Theorem~\ref{thm:unrestrictedEPPA} is a generalisation of the original 1992 result
of Hrushovski~\cite{hrushovski1992} for graphs and Herwig's strengthening to
relational structures~\cite{Herwig1995}. A variant for structures
with relations and unary function was shown by Evans, Hubi\v cka, Ne\v set\v ril
\cite{Evans3}, aiming to solve problems arising from the study of sparse graphs~\cite{Evans2}.
The notion of $\GammaL$-structures is motivated by a lemma on permomorphisms used
by Herwig~\cite[Lemma~1]{herwig1998}. It was noticed by Ivanov~\cite{Ivanov2015} that
this lemma is of an independent interest and can be used to show EPPA for classes
with definable equivalences.

There are multiple strategies for prove Theorem~\ref{thm:unrestrictedEPPA} for
relational structures, including group-theoretical ones (sometimes referred to as Mackey's
constructions) \cite{hrushovski1992,herwig2000,sabok2017automatic} and an easy combinatorial construction based on
intersection graphs~\cite{herwig2000}. The approach taken in our proof is new,
inspired by a related result of Hodkinson and Otto~\cite{hodkinson2003}.
We refer to it as a \emph{valuation construction}.
\end{remark}

\begin{remark}[History of unrestricted structural Ramsey results]
Generalising earlier results for colouring vertices~\cite{folkman1970,Nevsetvril1976b,Nevsetvril1977} and edges
of graphs, Theorem~\ref{thm:unrestrictedRamsey} for relational structures was proved by
Ne{\v{s}}et{\v{r}}il and R{\"o}dl in 1977~\cite{Nevsetvril1977} and, independently,
by Abramson and Harrington in 1978~\cite{Abramson1978}. A strengthening of this theorem
for classes of structures with relations and unary functions is relatively
easy to prove and was done in a special case by Hubi\v cka and Ne\v set\v ril~\cite{Hubivcka2014} (for bowtie-free graphs)
and independently by Soki{\'c}~\cite{Sokic2016} (for structures with unary functions only). The proof
strategies used in both of these papers turned out to be unnecessarily complex. For both Ramsey and EPPA, unary relations
and functions can be added by an incremental construction on top of an existing
Ramsey structure in a relational language. This general phenomenon is discussed
by Hubi\v cka and Ne\v set\v ril~\cite[Section 4.3.1]{Hubicka2016}.

The final, and substantial, strengthening was to introduce a construction for
structures in languages with function symbols of higher arities.
It is interesting to note that while there are multiple proofs of the unrestricted
theorem for structures in relational languages~\cite{Nevsetvril1989,PromelBook,Sauer2006}, the only known proof of Theorem~\ref{thm:unrestrictedRamsey}
uses a recursive variant of the \emph{partite construction}---presently the most
versatile method of constructing Ramsey objects developed by Ne\v set\v ril and R\"odl
in a series of papers since late 1970's~\cite{Nevsetvril1976,Nevsetvril1977,Nevsetvril1979,Nevsetvril1981,Nevsetvril1982,Nevsetvril1983,Nevsetvril1984,Nevsetvril1987,Nevsetvril1989,Nevsetvril1990,Nevsetvril2007}.
The recursive variant of the partite construction was introduced by Hubi\v cka and Ne\v set\v ril
to prove Theorem~\ref{thm:unrestrictedRamsey}. A special case was independently used
by Bhat, Ne\v set\v ril, Reiher and R\"odl to obtain the Ramsey property of the class
of all finite ordered partial Steiner systems~\cite{bhat2016ramsey}.
\end{remark}

\begin{remark}[Infinitary structural Ramsey theorems]
An infinitary variant of the unrestricted structural Ramsey theorem for graphs is shown by
Sauer~\cite{Sauer2006} (generalising work of Devlin~\cite{devlin1979}).  This is a highly non-trivial strengthening of the finitary version with
additional consequences for the automorphism groups~\cite{zucker2017}.
A generalisation for structures in a finite language with relations and
unary functions was recently claimed by Balko, Chodounsk\'y, Hubi\v cka,
Kone\v cn\'y and Vena~\cite{Hubickabigramsey}.
The first restricted theorems in this area were given recently by Dobrinen~\cite{dobrinen2017universal,dobrinen2019ramsey}.
Her proofs combine many techniques and are very technically challenging.
\end{remark}
\subsection{Sparsening constructions}
In order to work with classes of structures satisfying additional axioms (such as
metric spaces or triangle free graphs), it is usual to first apply the unrestricted theorems
(Theorems~\ref{thm:unrestrictedEPPA} and~\ref{thm:unrestrictedRamsey}) and then use
the resulting structures as a template to build bigger, and more sparse, structures with
the desired local properties.  
This is a nature of several earlier proofs of EPPA and the
Ramsey property~\cite{Nevsetvril2007,Dellamonica2012,herwig2000,hodkinson2003,Conant2015} and can be more systematically captured by the following
definition and theorems:
\begin{definition}[Tree amalgamation {\cite[Definition 7.1]{Hubicka2018EPPA}}]\label{defn:tree-amalgamation}
Let $\GammaL$ be a language equipped with a permutation group and let $\str A$ be a finite irreducible $\GammaL$-structure (Definition~\ref{def:irreducible}). We inductively define what a \emph{tree amalgamation of copies of $\str A$} is.
\begin{enumerate}
\item If $\str D$ is isomorphic to $\str A$ then $\str D$ is a tree amalgamation of copies of $\str A$.
\item If $\str B_1$ and $\str B_2$ are tree amalgamations of copies of $\str A$ and $\str D$ is a $\GammaL$-structure with an embedding to all of $\str A$, $\str B_1$ and $\str B_2$, then the free amalgamation of $\str B_1$ and $\str B_2$ over $\str D$ is also a tree amalgamation of copies of $\str A$.
\end{enumerate}
\end{definition} 
Recall the definition of homomorphism-embedding (Definition~\ref{def:homomorphism-embedding}).
In these terms, the sparsening constructions can be stated as follows.
\begin{theorem}[Hubi\v cka, Kone\v cn\' y, Ne\v set\v ril 2019]
\label{thm:sparseningEPPA}
Let $\GammaL$ be a finite language equipped with a permutation group where all function symbols are unary, $n\geq 1$, $\str{A}$ a finite irreducible $\GammaL$-structure and $\str{B}_0$ its finite EPPA-witness.
Then there exists a finite EPPA-witness $\str{B}$ for $\str A$ such that
\begin{enumerate}
 \item there is a homomorphism-embedding (a projection) $\str{B}\to \str{B}_0$,
 \item for every substructure $\str{B}'$ of $\str{B}$ with at most $n$ vertices there exists
a structure $\str{T}$ which is a tree amalgamation of copies of $\str{A}$ and a homomorphism-embedding $\str{B'}\to\str{T}$,
 \item $\str{B}$ is irreducible structure faithful.
\end{enumerate}
If $\str{B}_0$ is coherent, then $\str{B}$ is coherent, too.
\end{theorem}
\begin{theorem}[Hubi\v cka, Ne\v set\v ril 2019]
\label{thm:sparseningRamsey}
Let $L$ be a language, $n\geq 1$, $\str{A},\str{B}$ finite irreducible $L$-structures and $\str{C}_0$ a finite $L$-structure such that $\str{C}_0\longrightarrow(\str{B})^\str{A}_2$.
Then there exists a finite $L$-structure $\str{C}$ such that $\str{C}\longrightarrow (\str{B})^\str{A}_2$ and
\begin{enumerate}
 \item there exists a homomorphism-embedding $\str{C}\to \str{C}_0$,
 \item for every substructure $\str{C}'$ of $\str{C}$ with at most $n$ vertices there exists
a structure $\str{T}$ which is tree amalgamation of copies of $\str{A}$ and a homomorphism-embedding $\str{C}'\to\str{T}$,
 \item every irreducible substructure of $\str{C}$ is also a substructure of a copy of $\str{B}$ in $\str{C}$.
\end{enumerate}
\end{theorem}
While not stated in this form, Theorem~\ref{thm:sparseningEPPA} follows by a proof of Lemma 2.8 in~\cite{Hubicka2018EPPA} and Theorem~\ref{thm:sparseningRamsey} is a direct consequence of the iterated partite construction as used
in the proof of Lemmas 2.30 and 2.31~in~\cite{Hubicka2016}.
While the underlying combinatorics for EPPA and Ramsey constructions are very different, the overall structure of the proofs is the same.
To prove Theorem~\ref{thm:sparseningEPPA}, one repeats the valuation construction $n$ times, each time taking the result of the previous step as a template to build a new structure. Analogously, Theorem~\ref{thm:sparseningRamsey} is proved by repeating the partite construction $n$ times with a similar setup.

\subsection{Structural conditions on amalgamation classes}
Theorems~\ref{thm:sparseningEPPA} and~\ref{thm:sparseningRamsey} can be used as
``black boxes'' to show EPPA and the Ramsey property for almost all known
examples. However, to make their application easier, it is useful to introduce
some additional notions.

Recall that by Observations~\ref{obs:eppaamalgamation} and~\ref{obs:eppaamalgamation},
 all sensible candidates (that is, hereditary isomorphism-closed classes of finite structures with
the joint embedding property) for EPPA and the Ramsey property are amalgamation classes.
Since neither EPPA or the Ramsey property is implied by the amalgamation property in full generality (as can be demonstrated by counter-examples),
the aim of this section is to give structural conditions which are sufficient to
prove EPPA or the Ramsey property for a given amalgamation class $\mathcal K$.

\medskip

We will make a technical assumption that all structures in the class $\K$ we are working with are irreducible.
This will be helpful in formulating the conditions dealing with ``structures with holes''.
Irreducibility can be easily accomplished for any amalgamation class $\K$ by 
considering its expansion adding a
binary symbol $R$ and putting for every $(u,v)\in \rel{A}{}$
for every $\str{A}\in \K$ and every $u,v\in A$.

\medskip
The following concepts were introduced by Hubi\v cka and Ne\v set\v ril in the
Ramsey context~\cite{Hubicka2016} and subsequently adjusted for EPPA by Hubi\v cka,
Kone\v cn\' y and Ne\v set\v ril~\cite{Hubicka2018EPPA}.  We combine both
approaches.
\begin{figure}
\centering
\includegraphics{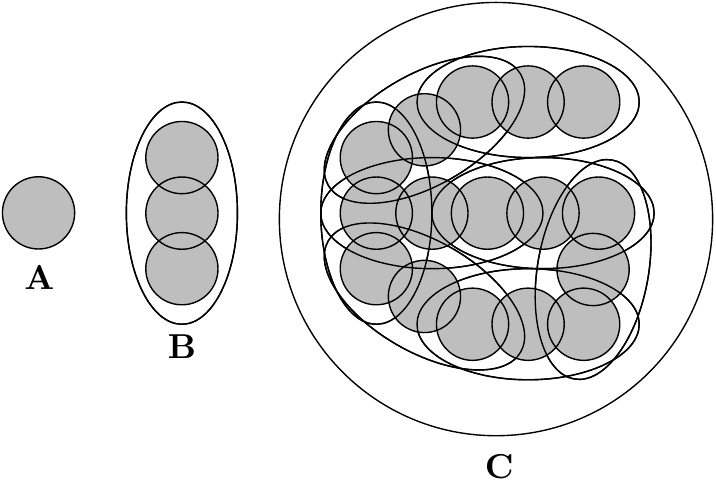}
\caption{Construction of a Ramsey object by multiamalgamation.}
\label{fig:multiamalgam}
\end{figure}
\begin{figure}
\centering
\includegraphics{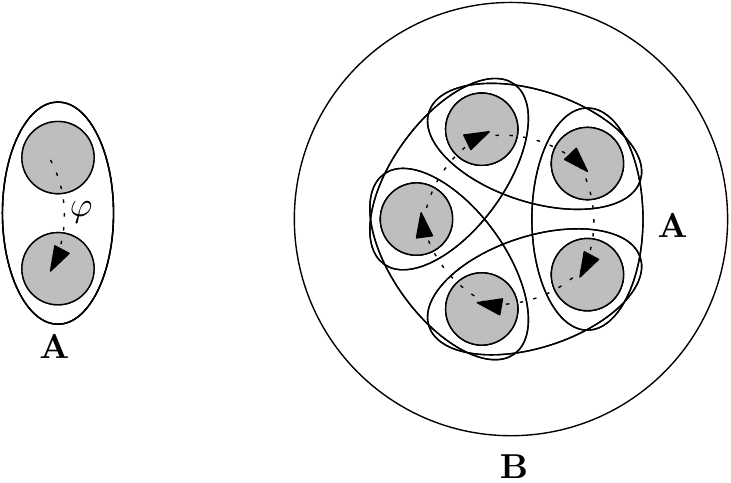}
\caption{Extending one partial automorphism $\varphi$.}
\label{fig:extending}
\end{figure}

At an intuitive level, it is not hard to see that
both Ramsey constructions and EPPA constructions can be seen as series of amalgamations all performed
at once as schematically depicted in Figures~\ref{fig:multiamalgam} and~\ref{fig:extending}.
These amalgamations must ``close cycles'' and can not be just tree amalgamations in the
sense of Definition~\ref{defn:tree-amalgamation}.

 Instead of working with complicated amalgamation diagrams, we  split the process into two steps---the \emph{construction} of the free multiamalgamation
 (which yields an incomplete, or ``partial'', structure) followed by a \emph{completion}
(cf. Bitterlich and Otto~\cite{bitterlich2019investigations}).

\begin{definition}[Completion {\cite[Definition 2.5]{Hubicka2016}, \cite[Definition 8.2]{Hubicka2018EPPA}}]
\label{defn:completion}
Let $\str{C}$ be a $\GammaL$-structure. An irreducible $\GammaL$-structure $\str{C}'$ is a \emph{completion}
of $\str{C}$ if there exists a homomorphism-embedding $\str{C}\to\str{C}'$.
If there is a homomorphism-embedding $\str{C}\to\str{C}'$ which is injective,
we call $\str{C}'$ a \emph{strong completion}.

We also say that a strong completion is \emph{automorphism-preserving} if for every $\alpha\in \Aut(\str C)$ there is $\alpha'\in \Aut(\str{C}')$ such that $\alpha\subseteq \alpha'$ and moreover the map $\alpha\mapsto\alpha'$ is a group homomorphism $\Aut(\str C)\to \Aut(\str C')$.

Of particular interest will be the question whether there exists a completion in a given class $\mathcal K$ of structures. In this case we speak about a \emph{$\mathcal K$-completion}.
\end{definition}

For classes of irreducible structures, (strong) completion may be seen as a generalisation of (strong) amalgamation: Let $\K$ be such a class. The (strong)
amalgamation property of $\K$ can be equivalently formulated as follows: For $\str{A}$, $\str{B}_1$, $\str{B}_2 \in \K$ and embeddings $\alpha_1\colon\str{A}\to\str{B}_1$ and $\alpha_2\colon\str{A}\to\str{B}_2$, there is a (strong) $\mathcal K$-completion
of the free amalgamation of $\str{B}_1$ and $\str{B}_2$ over $\str{A}$ with respect to $\alpha_1$ and $\alpha_2$.

Observe that the free amalgamation is not in $\K$ unless the situation is trivial.
Free amalgamation results in a reducible structure, as the pairs of vertices  where
one vertex belongs to $\str{B}_1\setminus \alpha_1(\str{A})$ and the other to $\str{B}_2\setminus \alpha_2(\str{A})$ are never both contained in a single tuple
of any relation. Such pairs can be thought of as  \emph{holes} and a completion is then a process of filling in the
holes to obtain irreducible structures while preserving all embeddings of irreducible structures.

\medskip

The key structural condition can now be formulated as follows:
\begin{definition}[Locally finite subclass {\cite[Definition 8.3]{Hubicka2018EPPA}, \cite[Definition 2.8]{Hubicka2016}}]\label{defn:locallyfinite}
Let $\mathcal E$ be a class of finite $\GammaL$-structures and $\mathcal K$ a subclass of $\mathcal E$ consisting of irreducible structures. We say
that the class $\mathcal K$ is a \emph{locally finite subclass of $\mathcal E$} if for every $\str A\in \mathcal K$ and every $\str{B}_0 \in \mathcal E$ there is a finite integer $n = n(\str A, \str {B}_0)$ such that 
every $\GammaL$-structure $\str B$ has a completion $\str B'\in \mathcal K$ provided that it satisfies the following:
\begin{enumerate}
\item every irreducible substructure of $\str{B}$ lies in a copy of $\str A$,
\item there is a homomorphism-embedding from $\str{B}$ to $\str{B}_0$, and,
\item every substructure of $\str{B}$ on at most $n$ vertices has a comple\-tion in $\mathcal K$.
\end{enumerate}
We say that $\mathcal K$ is a \emph{locally finite automorphism-preserving subclass of $\mathcal E$} if $\str B'$ can always be chosen to be strong and automorphism-preserving.
\end{definition}

The following results are our main tools for obtaining EPPA and Ramsey results.
(And are, in fact, corollaries of Theorems~\ref{thm:sparseningEPPA} and \ref{thm:sparseningRamsey}.)

\begin{theorem}[Hubi\v cka, Kone\v cn\' y, Ne\v set\v ril~\cite{Hubicka2018EPPA}]
\label{thm:mainstrongEPPA}
Let $\GammaL$ be a finite language equipped with a permutation group where all function symbols are unary (no restrictions are given on the relational symbols) and
let $\mathcal E$ be a class of finite irreducible $\GammaL$-structures with EPPA. Let $\K$ be a hereditary
locally finite automorphism-preserving subclass of $\mathcal E$ with the strong amalgamation property. Then $\K$ has EPPA.

  Moreover, if EPPA-witnesses in $\mathcal E$ can be chosen to be coherent then EPPA-witnesses in $\K$ can be chosen to be coherent, too.
\end{theorem}
\begin{theorem}[Hubi\v cka, Ne\v set\v ril~\cite{Hubicka2016}]
\label{thm:mainstrong}
Let $L$ be a language, let $\mathcal R$ be a Ramsey class of irreducible finite $L$-structures and let $\K$ be a hereditary
locally finite subclass of $\mathcal R$ with the strong amalgamation property.
Then $\K$ is a Ramsey class.

Explicitly:
For every pair of struc\-tures $\str{A}, \str{B}\in\K$ there exists
a structure $\str{C} \in \K$  such that
$$
\str{C} \longrightarrow (\str{B})^{\str{A}}_2.
$$
\end{theorem}

For applications, it is important that in many cases the existence of $\K$-completions and strong $\K$-completions coincide. This can be formulated as follows.
\begin{prop}[{Hubi\v cka, Ne\v set\v ril~\cite[Proposition 2.6]{Hubicka2016}}]
\label{prop:strongcompletion}
Let $L$ be a language such that all function symbols in $L$ have arity one (there is no restriction on relational symbols) and let $\K$ be a hereditary class of finite irreducible $L$-structures with the strong amalgamation property.
A finite $L$-structure $\str{A}$ has a $\K$-completion if and only if it has a strong $\K$-completion.
\end{prop}
With this proposition at hand, verification of the condition given by Definition~\ref{defn:locallyfinite} can be carried out for many amalgamation classes.
Several examples aare worked out in~\cite[Section 4]{Hubicka2016} and followup papers (see, for example, \cite{Sam,Hubicka2017sauerconnant,Aranda2017,Aranda2017c,Konecny2018bc,Konecny2018b,Konecny2019a}).

Let us discuss two examples which demonstrate the techniques for verifying local finiteness and
also a simple situation where this condition does not hold.
\begin{example}[{Metric spaces with distances $1$, $2$, $3$ and $4$~\cite[Example 2.9]{Hubicka2016}}]
\label{example:metric}
Consider a language $L$ containing binary relations $\rel{}{1}$, $\rel{}{2}$, $\rel{}{3}$ and $\rel{}{4}$ which we understand as
distances. Let $\mathcal E$ be the class of all irreducible finite structures where all four relations are symmetric,
irreflexive and every pair of distinct vertices is in precisely one of relations $\rel{}{1}$, $\rel{}{2}$, $\rel{}{3}$, or $\rel{}{4}$ ($\mathcal E$ may be viewed a class of 4-edge-coloured complete graphs). Let $\mathcal K$ be a subclass of $\mathcal E$ consisting of those structures which satisfy
the triangle inequality (\ie{}\ $\mathcal K$ is the class of finite metric spaces with distances 1, 2, 3, and 4).

It is not hard to verify that an $L$-structure $\str{B}$ which has a completion to some $\str{B}_0\in  \mathcal E$ (meaning that all relations are
symmetric and irreflexive and every pair of distinct vertices is in at most one relation) can be
completed to a metric space if and only if it contains no non-metric triangles (\ie{}\ triangles with
distances 1--1--3, 1--1--4 or 1--2--4) and no 4-cycle with distances 1--1--1--4, see Figure~\ref{fig:4cycles}.
This can be done by computing the shortest distance among the edges present in the partial structure.
Such completion is also clearly automorphism-preserving.
\begin{figure}
\centering
\includegraphics{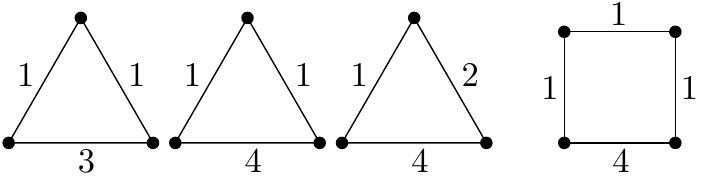}
\caption{Cycles with no completion to a metric space with distances $1,2,3$ and $4$.}
\label{fig:4cycles}
\end{figure}
 It follows that $\mathcal K$ is a locally
finite subclass of $\mathcal E$ and for every $\str{C}_0 \in \mathcal E$ we can put $n(\str{C}_0) = 4$.
\end{example}
\begin{example}[{Metric spaces with distances $1$ and $3$~\cite[Example 2.10]{Hubicka2016}}]
\label{example:13b}
Now consider the class $\mathcal K_{1,3}$ of all metric spaces which use only distances one and three. It is easy to see
that $\mathcal K_{1,3}$ is not a locally finite subclass of $\mathcal E$ (as given in Example~\ref{example:metric}). To see that let $\str T\in \mathcal E$ be the triangle with distances 1--1--3. Now consider a cycle $\str C_n$ of length $n$ with one edge of distance three and the others of distance one (as depicted in Figure~\ref{ultra}).
\begin{figure}
\centering
\includegraphics{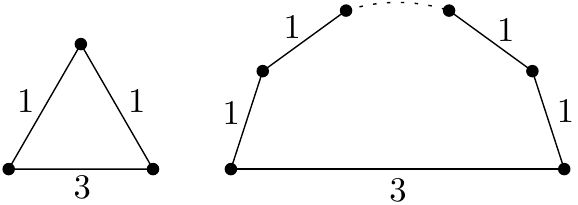}
\caption{Cycles with no completion to a metric spaces with distances one and three.}
\label{ultra}
\end{figure}
$\str{T}$ is a completion of $\str C_n$, however it has no $\mathcal K_{1,3}$-completion. Moreover, every proper substructure of $\str C_n$ (that is, a path consisting of
at most one edge of distance three and others of distance one) does have a $\mathcal K_{1,3}$-completion.
It follows that there is no $n(\str{T})$ and thus $\mathcal K_{1,3}$ is not a locally finite subclass of $\mathcal E$.
\end{example}
We remark that an equivalent of Proposition~\ref{prop:strongcompletion} does
not hold for languages with functions of greater arity. This can be demonstrated
on a homogenization of the class of all finite graphs of girth 5~\cite[Example 2.7]{Hubicka2016}.
\section{Key examples}
\label{sec:examples}
Thanks to the general constructions presented in Section~\ref{sec:general} and
systematic work on the classification programme, the landscape of known Ramsey
and EPPA classes has recently changed significantly. Instead of several
isolated examples, we nowadays know many classes which are too numerous to be
fully covered by this introduction. We however list those which we consider
most significant for developing the general theory.

\subsection{Free amalgamation classes}
The Ne\v set\v ril--R\"odl theorem~\cite{Nevsetvril1977} was the fist result which
gave the Ramsey property of a class of structures satisfying additional axioms.
Let us cite its 1989 formulation~\cite{Nevsetvril1989}.
\begin{theorem}[Ne\v set\v ril--R\"odl theorem for hypergraphs~\cite{Nevsetvril1989}]
\label{thm:NRoriginal}
Let $\str{A}$ and $\str{B}$ be ordered hypergraphs, then there exists an ordered hypergraph
$\str{C}$ such that $\str{C}\longrightarrow (\str{B})^\str{A}_2$.

Moreover, if $\str{A}$ and $\str{B}$ do not contain an irreducible hypergraph
$\str{F}$ (as a non-induced sub-hypergraph) then $\str{C}$ may be chosen
with the same property.
\end{theorem}
In this formulation, it does not speak about an amalgamation class of structures,
but it is not hard to work out that the original partite construction proof
works in fact for all free amalgamation classes of relational structures
expanded by a free linear order. This may be seen as an interesting coincidence,
since the partite construction is not based only on a free amalgamation argument only.
However, the connection is given by the following easy observation:
\begin{observation}
For every free amalgamation class $\K$ of $\GammaL$-struc\-tures there exists a family
of irreducible structures $\mathcal{F}$ such that $\K$ is precisely the class of
all finite structures $\str{A}$ for which there is no $\str{F}\in \mathcal F$
which embeds into $\str{A}$.
\end{observation}
\begin{proof}
Put $\mathcal F$ to be all $\GammaL$-structures $\str{F}$ such that
$\str{F}\notin \mathcal K$ and every proper substructure of $\str F$ is in $\mathcal K$.
It follows that all such structures $\str{F}$ are irreducible (otherwise
one obtains a contradiction with $\K$ being a free amalgamation class) and
has the desired property.
\end{proof}
This can be re-formulated as follows.
Given a free amalgamation class $\K$ and a structure $\str{B}\in \K$,
we know that every structure $\str{C}$ such that every irreducible substructure
of $\str{C}$ is also a substructure of $\str{B}$ is in the class $\K$.
This is precisely what the original formulation of the Ne\v set\v ril--R\"odl theorem
gives. In the EPPA context, a stronger condition is accomplished by irreducible
structure faithfulness.

It is thus well established that all free amalgamation classes of relational
structures are Ramsey when expanded with a free linear order and this expansion has the expansion (ordering) property (see \eg~Ne\v set\v ril's~\cite{Nevsetvril1995} or Nguyen Van Th\'e's~\cite{NVT14} surveys). All free amalgamation
classes of relational structures also have (coherent) EPPA by an analogous result of Hodkinson and Otto~\cite{hodkinson2003},
as observed by Siniora~\cite{Siniora2} (with coherence verified by Siniora and Solecki~\cite{Siniora}).

Note that the class $\mathcal F$ can also be seen as a class of minimal \emph{obstacles}.
 The
analysis of obstacles in a generalised sense of
Definition~\ref{defn:locallyfinite} remains the main tool for obtaining
Ramsey and EPPA results.

\medskip

The situation is more complicated for languages involving functions.
Forty years after the Ne\v set\v ril--R\"odl theorem,
Evans, Hubi\v cka and Ne\v set\v ril~\cite[Theorem 1.3]{Evans3} showed that all free amalgamation classes
are Ramsey when expanded by a free linear order. However, the expansion property does not
necessarily hold~\cite[Proposition 3.1]{Evans3}. The special (admissible) orderings
which lead to a Ramsey class with the expansion property are explicitly described
in Section~3.1 of~\cite{Evans3}. 

A coherent EPPA theorem for free amalgamation classes with unary functions was also shown by Evans, Hubi\v cka and Ne\v set\v ril~\cite[Theorem 1.3]{Evans3}.
A generalisation for languages with non-unary functions remains open.

\medskip

Ramsey expansions of free amalgamation classes also follow by a direct application of Theorems~\ref{thm:unrestrictedRamsey} and~\ref{thm:mainstrong} for $n=1$.
Similarly, a generalisation of~\cite[Theorem 1.3]{Evans3} for $\GammaL$-structures in languages with unary functions follows by a direct application of Theorems~\ref{thm:unrestrictedEPPA} and~\ref{thm:mainstrongEPPA} for $n=1$.
A special case of an EPPA construction for free amalgamation class with non-unary functions is discussed in~\cite[Section 9.3]{Hubicka2018EPPA}. It makes a
non-trivial use of the permutation group on the language, where the result of one EPPA construction defines the language and its permutation group for another EPPA construction. In general, however, the question of EPPA for free amalgamation classes in languages involving non-unary functions remains open.

\subsection{Partial orders}
\label{sec:partialorders}
The class of all partial orders expanded by a linear extension forms a Ramsey
class.  This result was announced by Ne\v set\v ril and R\"odl in
1984~\cite{Nevsetvril1984} and the first proof was published by Paoli, Trotter and
Walker one year later~\cite{Trotter1985}. The original proof of Ne\v set\v ril and R\"odl using the partite construction
was published only recently~\cite{nevsetvril2018ramsey}. Several alternative proofs
are known~\cite{sokic2012ramsey,masulovic2016pre}.

This is the first example where an interaction between the partial order in the class
and a linear order of its expansion was observed (a phenomenon that generalises
to other examples of reducts of partial~\cite{pach2014reducts} and total orders~\cite{junker2008116}).
  It can be easily shown
that partial orders do not form a Ramsey class when expanded by a free linear
order~\cite[Lemma 4]{sokic2012ramsey}.

The Ramsey property for partial orders with linear extensions can be shown by
an easy application of Theorem~\ref{thm:mainstrong}~\cite[Section 4.2.1]{Hubicka2016}.
In fact, this class served as a key motivation
for the notion of locally finite subclass. As discussed in~\cite[Section 4.2.1]{Hubicka2016},
the choice of parameter $n$ in Definition~\ref{defn:locallyfinite} depends on
the number of vertices of $\str{B}_0$.  It most other applications, the bound $n$
can be chosen globally for all possible choices of $\str{B}_0$. This
example demonstrates the power of the ``ambient linear order'' present in every 
Ramsey class.  

Observe that the class of all partial order does not have EPPA for the same
reason as in the case of total orders. This example thus shows the importance
of the ``base class'' $\mathcal E$ in the definition of locally finite subclass.

\subsection{Metric spaces}
\label{sec:metric}
The Ramsey property of linearly ordered metric spaces was shown by Ne\v set\v ril~\cite{Nevsetvril2007}.  He isolated the general form of the iterated
partite construction which is also the basic mechanism of the proof of
Theorem~\ref{thm:sparseningRamsey} (special cases of this technique have been used
since late 1970's~\cite{Nevsetvril1979}, but here it appears in a very general form). 
A related construction was also independently used by Dellamonica and R{\"o}dl for
the class of graphs with distance preserving embeddings~\cite{Dellamonica2012}.

The iterated partite construction has proven to be a useful tool for obtaining many additional
results. However, quite surprisingly, a rather simple reduction to the (believed to be
much easier) Ramsey property of the class of finite partial orders with a linear
extension has recently been published by Ma{\v{s}}ulovi{\'c}~\cite{masulovic2016pre}.  This
construction does not however seem to extend to other cases on which the
iterated partite construction method can be applied.

\medskip

Metric spaces presented an important example in the study of EPPA classes, too.
EPPA for metric spaces was obtained independently by Solecki~\cite{solecki2005}
and Vershik~\cite{vershik2008}. Solecki's proof is among the first applications
of the deep Herwig--Lascar theorem~\cite{herwig2000}, while Vershik announced a direct
proof which remains unpublished. Additional proofs were published by Pestov~\cite{Pestov2008}, Rosendal~\cite{rosendal2011}, Sabok~\cite[Theorem 8.3]{sabok2017automatic}. All these proofs
are based on group-theoretical methods (the M.~Hall theorem~\cite{hall1949}, the Herwig--Lascar
theorem~\cite{herwig2000,otto2017,Siniora}, the {R}ibes--{Z}alesski{\u\i} theorem \cite{Ribes1993} or Mackey's
construction~\cite{mackey1966}). A simple combinatorial proof was found by Hubi\v cka, Kone{\v c}n{\'y and Ne{\v{s}}et{\v{r}}il in 2018~\cite{Hubicka2018metricEPPA},
the core ideas of which were later developed to the form of Theorem~\ref{thm:sparseningEPPA}.

\medskip
Metric spaces with additional axioms are still presenting interesting challenges in
the classification programme.
Several special classes of metric spaces are considered by Nguyen Van Th{\'e} in his
monograph~\cite{The2010}. A generalisation of metric spaces was considered by Conant~\cite{Conant2015}.

Analysis of local finiteness based on non-metric cycles is outlined in Example~\ref{example:metric} and can
be extended to various restricted cases of generalised and restricted metric spaces. This presents a currently active line of research~\cite[Section 4.2.2]{Hubicka2016}, \cite{Aranda2017c,Hubickabigramsey,Hubicka2017sauerconnant,Hubicka2017sauer,Konecny2018bc,Konecny2018b}.
One of the important open questions in the area is the existence of a precompact Ramsey expansion (or EPPA)
of the class of all finite affinely independent Euclidean metric spaces~\cite{The2010}.

\subsection{Homogenisations of classes defined by forbidden homomorphic images}
Fix a family $\F$ of finite connected $L$-structures and consider the class
$\K_\mathcal F$ of all countable structures $\str{A}$ such that there is no $\str{F}\in
\F$ with a monomorphism to $\str{A}$. (For graphs this means that $\str{F}$ is isomorphic no subgraph $\str{S}$ of $\str{A}$. Note that here $\str{S}$ is not necessarily induced.)
In 1999, Cherlin, Shelah and Shi~\cite{Cherlin1999} gave a structural condition for the existence of
a \emph{universal} structure $\str{U}\in K_\mathcal F$, that is a structure which contains a copy
of every other structure in $\K_\mathcal F$.

Deciding about existence of a universal structure for a given class $\K_\mathcal F$
is a non-trivial task even for $\F$ consisting of a single graph~\cite{Cherlin2001,Cherlin2007,Cherlin2007a,Cherlin2007b,Cherlin2011,Cherlin2015,Cherlinb}.
However when $\F$ is finite and closed for homomorphism-embedding images then a universal structure always exists.
Hubi\v cka and Ne\v set\v ril \cite{Hubicka2009,Hubicka2013}
 studied explicit constructions of 
universal structures for such classes $\K_\F$.
(In earlier works, homomorphisms are used in place of homomorphism-em\-bedd\-ings.)
This led to an explicit construction which expands
class $\Forb(\F)$ to an amalgamation class by means of new relations (called a \emph{homogenization}
and first used in this context by
Covington~\cite{Covington1990}) and 
the universal structure is then the reduct of the \Fraisse{} limit of $\Forb(\F)$.

Work on a general theorem giving Ramsey expansions for $\Forb(\F)$ classes
resulted in an unexpected generalisation  in the form of Theorem~\ref{thm:mainstrong}.
For that Proposition~\ref{prop:strongcompletion} giving close interaction between the strong amalgamation and $\Forb(\F)$ classes was necessary.
The
relationship to homogenisations is exploited in~\cite[Section 3]{Hubicka2016} 
and an explicit description of the expansion is given. 
These results were recently applied on infinite-domain constraint satisfaction
problems by Bodirsky, Madelaine and Mottet~\cite{bodirsky2018universal}.

Curiously, a similar general result in the EPPA context was formulated a lot
earlier in the form of the Herwig--Lascar theorem~\cite{herwig2000}. 
One of the contributions of Theorem~\ref{thm:mainstrongEPPA}
is thus giving a generalisation of the
Herwig--Lascar theorem with a new, and more systematic, proof. The Herwig--Lascar theorem
has been generally regarded as the deepest result
in the area. See also~\cite{bitterlich2019investigations}.

\subsection{Classes with algebraic closures}
\label{sec:closures}
Amalgamation classes which are not strong amalgamation classes lead to \Fraisse{} limits with
non-trivial algebraic closure:
\begin{definition}
\label{def:algebraic}
Let $L$ be a relational language, let $\str{A}$ be an $L$-structure and let $S$ be a finite subset of $A$.  The
\emph{algebraic closure of $S$ in $\str{A}$}, denoted by $\ACl_\str{A}(S)$, is the
set of all vertices $v\in A$ for which there is a formula $\phi$ in the language $L$ with $\vert S\vert +1$
variables such that $\phi(\vv{S},v)$ is true and there are only finitely many
vertices $v'\in A$ such that $\phi(\vv{S},v')$ is also true. (Here, $\vv{S}$
is a fixed enumeration of $S$.)
\end{definition}
To our best knowledge, the first class with non-trivial algebraic closure studied
in the direct structural Ramsey context was the class of all finite bowtie-free
graphs~\cite{Hubivcka2014} (inspired by the aforementioned work of Cherlin,
Shelah and Shi~\cite{Cherlin1999} and an earlier result of Komj\'ath, Mekler and
Pach~\cite{Komjath1988}).

While this class was originally thought to be a good candidate for a class with no
precompact Ramsey expansion, this turned out to not be the case.  In the
direction of finding a proof of the Ramsey property for this class, the partite construction was
extended for languages involving functions, which can be used to connect a set
to its closure.  This technique turned out to be very general and is
discussed in detail in~\cite[Section~4.3]{Hubicka2016}.

EPPA for the class of all finite bowtie-free graphs was considered by
Siniora \cite{siniora2017bowtie,Siniora2} who gave a partial result for
extending one partial automorphisms.  Thanks to the general results on free
amalgamation classes with unary functions the proof of the existence of a Ramsey expansion and
of EPPA for the class of all finite bowtie-free graphs is now
easy~\cite[Setion 5.4]{Evans3} and in fact generalises to all known Cherlin--Shelah--Shi classes with one constraint~\cite[Section 4.4.2]{Hubicka2016}.

\medskip
Introducing expansions with function symbols has turned out to be a useful tool
for many additional examples which we outline next.

\subsection{Classes with definable equivalences}
\label{sec:equivalences}
Example~\ref{example:13b} showing a subclass which is not a locally finite subclass
can be generalised using the model-theoretical notion of definable equivalences.

Let $\str{A}$ be an $L$-structure.  An \emph{equivalence formula} is a
first order formula $\phi(\vv{x},\vv{y})$ which is symmetric and transitive on the set of
all $n$-tuples $\vv{a}$ of vertices of $\str{A}$ where
$\phi(\vv{a},\vv{a})$ holds (the set of such $n$-tuples is called the \emph{domain} of
the equivalence formula $\phi$).

It is not hard to observe that definable equivalences with finitely many
equivalence classes may be obstacles to being Ramsey.  For a given ordered
structure $\str{U}$, we say that $\phi$ is an \emph{equivalence formula on
copies of $\str{A}$} if $\phi$ is an equivalence formula,
$\phi(\vv{a},\vv{a})$ holds if and only if the structure induced by
$\str{U}$ on $\vv{a}$ is isomorphic to $\str{A}$ (in some fixed order of vertices of $\str{A}$).

\begin{prop}[{\cite[Proposition 4.25]{Hubicka2016}}]
\label{prop:imaginary}Let $\mathcal K$ be a hereditary Ramsey class of ordered $L$-structures, $\str{U}$ its \Fraisse{} limit,
$\str{A}$ a finite substructure of $\str{U}$ and $\phi$ an equivalence formula on copies of $\str{A}$.
Then $\phi$ has either one or infinitely many equivalence classes.
\end{prop}
\begin{proof}
Assume to the contrary that $\phi$ is an equivalence formula on copies of $\str{A}$ which defines $k$ equivalence classes, $1<k<\infty$.
 It is well known that from ultrahomogeneity we can assume that $\phi$ is quantifier-free.
Consequently, there is a finite substructure $\str{B}\subseteq \str U$ containing two copies of $\str{A}$ which belong to two different equivalence classes of $\phi$.
Partition $\str U\choose \str A$ to $k$ equivalence classes of $\phi$. Since $\phi$ is quantifier free, we get that there is no $\widetilde{\str B}\in {\str U \choose \str B}$ such that $\widetilde{\str B}\choose \str A$ would lie in a single equivalence class. Clearly, this implies that there is no $\str C\in \mathcal K$ such that $\str{C}\longrightarrow (\str{B})^\str{A}_k$, hence contradicting the Ramsey property.
\end{proof}
In model theory,
 an \emph{imaginary} element $\vv{a}/\phi$ of
$\str{A}$ is an equivalence formula $\phi$ together with a representative
$\vv{a}$ of some equivalence class of $\phi$.
Structure $\str{A}$ \emph{eliminates  imaginary $\vv{a}/\phi$} if there
exists a first order formula $\Phi(\vv{x},\vv{y})$ such that there is a unique tuple $\vv{b}$
such that $\phi(\vv{x}, \vv{a}) \iff \Phi(\vv{x}, \vv{b})$.

The notion of elimination of imaginaries separates those definable equivalence
which are not a problem for Ramsey and EPPA constructions (for example, graphs
have a definable equivalence on pairs of vertices with two equivalence classes:
edges and non-edges) and those which are problem (such as balls of diameter
one in Example~\ref{example:metric}). 

To give a Ramsey expansion, we thus first want to expand the structure and achieve
elimination of (relevant) imaginaries.
For a given equivalence formula $\phi$ with finitely many equivalence classes,
it is possible to expand the language by explicitly adding relations representing the
individual equivalence classes. For an equivalence formula $\phi$ with infinitely
many equivalence classes, the elimination of imagines can be accomplished by
adding new vertices representing the equivalence classes and a function symbol
linking the elements of individual equivalence classes to their representative.

This technique was first used to obtain Ramsey expansion of $S$-metric
spaces~\cite[Section 4.3.2]{Hubicka2016}.  Braunfeld used this technique to
construct an interesting Ramsey expansion of $\Lambda$-ultrametric spaces~\cite{Sam} (here the Ramsey expansion does not consist of a single
linear order of vertices but by multiple partial orders), which was later generalised in Kone\v cn\' y's master thesis~\cite{Konecny2018b}.

\medskip

Analogous problems need to be solved for the EPPA constructions as well.  In the case
of equivalences on vertices, this can be accomplished by unary
functions in the same was as for Ramsey expansion. A technique to work with equivalence
on $n$-tuples with infinitely many equivalence classes was introduced by
Ivanov~\cite{Ivanov2015} who observed that a technical lemma on
``permomorphisms''~\cite[Lemma~1]{herwig1998} can be used to obtain EPPA here.
Hubi\v cka, Kone\v cn\'y and Ne\v set\v ril~\cite{Hubicka2018EPPA} generalised
this technique to the notion of $\GammaL$-structures
which in turn proved to be useful for more involved constructions.

EPPA for the class of all structures with one quaternary relation defining equivalences
on subsets of size two with two equivalence classes is an open problem \cite{eppatwographs}.

\subsection{Classes with non-trivial expansions}
Ramsey expansions are not always obtained by ordering vertices and
imaginaries alone.  An early example (generic local order) is already discussed in Example~\ref{example:S2}.
Another example of this phenomenon is the class of two-graphs~\cite{Seidel1973}.
\emph{Two-graph $\str{H}$, corresponding to a graph $\str{G}$} is a 3-uniform hypergraph created from $\str G$ by putting
a hyper-edge on every tripe containing an odd number of edges. 

While there are non-isomorphic graphs such that their two-graphs are isomorphic,
it is not hard to show that every Ramsey expansion of two-graphs must fix one
particular graph~\cite[Section 7]{eppatwographs}. This means that the class of
all finite two-graphs (that is, finite hypergraphs which are two-graphs corresponding to some
finite graph $\str{G}$) is not a locally finite subclass of the class of all
finite hypergraphs.  Yet, quite surprisingly, this class has EPPA as was
recently shown by Evans, Hubi\v cka, Ne\v set\v ril and Kone\v cn\'y~\cite{eppatwographs}.
As discussed in Section~\ref{sec:classification} presents interesting example with
respect to question \ref{Q4} (where the answer is open but conjectured to be negative) and \ref{Q5}
(where answer is negative).

We believe that this class is the first known example of a class with EPPA but not ample generics~\cite[Section 8]{eppatwographs}.
It is also presently open if this class has coherent EPPA.

\medskip

Yet another related example is the class of all semigeneric
tournaments. A Ramsey expansion was given by Jasi{\'n}ski, Laflamme, Nguyen
Van Th{\'e} and Woodrow \cite{Jasinski2013}.  EPPA was recently claimed by
Hubi\v cka, Jahel, Kone\v cn\'y and Sabok \cite{HubickaSemigenericAMUC}.

\medskip

These examples share the property that there is no auto\-morphism-preserving
completion property, yet EPPA follows by some form of the valuation construction.
This suggests that Theorem~\ref{thm:mainstrongEPPA} can be strengthened. It
is however not clear what the proper formulation should be.

Most of these examples can be seen as reducts of other homogeneous structures
and the valuation constructions are related to this process.
There is an ongoing classification programme of such reducts which makes
use of knowing the corresponding Ramsey classes (see \eg{}~\cite{Thomas1991,thomas1996reducts,junker2008116,bodirsky2010reducts,Bodirsky2011,bodirsky201542,pach2014reducts}).
This may close a full circle: the study of Ramsey properties leads to a better understanding
of reducts which, in turn, may lead to a better understanding of EPPA.
Some open problems in this direction are discussed in~\cite{eppatwographs}.

\medskip

Perhaps the most surprising example was however identified by Evans and is
detailed in Section~\ref{sec:negative}.

\subsection{Limitations of general methods}
Current general theorems seems to provide a very robust and systematic
framework for obtaining Ramsey expansions and EPPA. We refer the reader to~\cite[Section 4]{Hubicka2016}
which provides numerous additional positive examples.
However, there are known
examples which do not follow by these techniques:

\begin{enumerate}
\item The class of all finite groups is known to have EPPA~\cite{Siniora,song2017hall}.
It is not clear what the corresponding Ramsey results should be.
Moreover, EPPA does not seem to follow by application of Theorem~\ref{thm:mainstrongEPPA}.
\item The class of all skew-symmetric bilinear forms has EPPA~\cite{Evans2005}.
Again it is now known if there is a precompact relational Ramsey expansion and Theorem~\ref{thm:mainstrongEPPA} does not seem to apply.
\item The Graham--Rothschild theorem implies that the class of finite boolean algebras is Ramsey.
In general, dual Ramsey theorems seems to not be covered by the presented framework, including
structural dual Ramsey theorems by Pr{\"o}mel~\cite{promel1985induced}, Frankl, Graham and R\"odl~\cite{frankl1987} and Solecki \cite{Solecki2010}.
These results all share similar nature but differ in the underlying categories.
It seems to the author that proper foundations for the structural Ramsey theory of dual structures need to be developed
building on non-structural results~\cite{Leeb,Graham1972,spencer1979ramsey}, approaches based on category theory~\cite[Section 2.3]{promel1985induced}
and recent progress on the projective \Fraisse{} limits~\cite{irwin2006,panagiotopoulos2016compact}.
It is also an interesting question, how to relate this to the self-dual Ramsey theorem
of Solecki~\cite{Solecki2013,Solecki2014}.

This line of research seems promising.
Barto\v sov\'a and Kwiatkowska applied the KPT-correspondence on the Lelek fans~\cite{bartovsova2014,bartovsova2017universal}.
A dualisation of Theorem~\ref{thm:mainstrong} is currently work in progress.
Dualisation of EPPA is also a possible future line of research.

\item Melleray and Tsankov adapted the KPT-correspondence to the context of metric structures~\cite{melleray2014extremely}.
First successful application of Ramsey theory in this direction was done by Barto\v sov\'a, Lopez-Abad, Lupini and Mbombo~\cite{bartosova2017ramsey}.
\item Soki{\'c} has shown Ramsey expansions of the class of lattices~\cite{sokic2015semilattices}. Again it is not clear
how to obtain this result by application of Theorem~\ref{thm:mainstrong}.
\end{enumerate}

\section{Negative results}
\label{sec:negative}
We conclude this introduction by a brief review of the surprising negative result of \cite{Evans2}.

Given the progress of the classification programme, one may ask if there are negative answers to
question~\ref{Q1} (Section~\ref{sec:classification}). 
More specifically, is there a homogeneous structure $\str{H}$ such that
there is no Ramsey structure $\str{H}^+$ which is a precompact relational expansion of $\str{H}$?

Eliminating easy counter-examples (such as $\mathbb Z$ seen as a metric space~\cite{NVT14}), a more specific question was raised by
Melleray, Nguyen Van Th{\'e} and Tsankov \cite[Question 1.1]{Melleray2015}, asking if the answer to question~\ref{Q1}
is positive for every $\omega$-categorical homogeneous structure $\str{H}$ (that is, it's automorphism group has only finitely many orbits on $n$-tuples for every $n$).
In a more restricted form, Bodirsky, Pinsker and Tsankov~\cite{Bodirsky2011a} asked if the answer to question~$\ref{Q1}$ is positive for every structure $\str{H}$
homogeneous in a finite relational language.

A negative answer to the question asked by Melleray, Nguyen Van Th{\'e} and Tsankov was given by Evans~\cite{Evans} (the latter question remains open). This led to the following:
\begin{theorem}[Evans, Hubi\v cka, Ne\v set\v ril~\cite{Evans2}] \label{cex} There exists a countable, $\omega$-categorical $L$-structure $\str{H}$ with the property that if $\Gamma \leq \Aut(\str{H})$ is extremely amenable (in other words, $\Gamma$ is an automorphism group of a Ramsey structure which is an expansion of $\str{H}$), then $\Gamma$ has infinitely many orbits on $\str{H}^2$. In particular, there is no precompact expansion of $\str{H}$ whose automorphism group is extremely amenable.
\end{theorem}

\medskip

This is demonstrated on a specific structure $\str{H}$ constructed using the Hrushovski predimension construction.
The automorphism group of $\str{H}$ has many interesting properties~(see \cite{Evans2}). We single out the fact that,
unlike the case of integers (where the only possible Ramsey expansion is the trivial one naming every vertex),
there exists a non-trivial Ramsey (non-precompact) expansion $\str{H}^+$ of~$\str{H}$.
In this context, however, the expansion property can not be used to settle that $\str{H}^+$
is ``minimal'' in a some sense. Answering a question asked by Tsankov in Banff meeting in 2015,  \cite{Evans2} shows
that $\Aut(\str{H}^+)$ is maximal among extremely amenable subgroups of $\Aut(\str{H})$.
This justifies the minimality of this Ramsey expansion.  The question about uniqueness remains open.

A similar situation arises with EPPA. The age of $\str{H}$ does not have EPPA~\cite[Corollary 3.9]{Evans2} (thus
the answer to \ref{Q3} is negative).  However, a
meaningful EPPA expansion is proposed in the form of an amenable
subgroup~\cite[Theorem 6.11]{Evans2}. This subgroup is again conjectured to be
maximal among the amenable subgroups of
$\Aut(\str{H})$~\cite[Conjecture 7.5]{Evans2}.  Hubi\v cka, Kone\v cn\'y and
Ne\v set\v ril recently proved EPPA for this class~\cite[Theorem
9.8]{Hubicka2018metricEPPA}. The second part of \cite[Conjecture 7.5]{Evans2} (about maximality)
remains open.

These examples are not isolated and several variants of this construction
are considered~\cite{Evans2}. 
This shows that the interplay between EPPA,
Ramsey classes and amalgamation classes is more subtle than previously believed, especially in the
context of structures in languages with functions.
Because most of structural Ramsey theory was developed in the context
of relational structures, these situations have only been encountered recently.
These results also demonstrate that it is meaningful (and very interesting) to extend the classification programme even for structures with such
behaviour.

\section*{Acknowledgement}
The author is grateful to David Evans, Gregory Cherlin and Mat\v ej Kone\v cn\'y for number of remarks
and suggestions which improved presentation of this paper.
\bibliographystyle{alpha}
\bibliography{./bibliography.bib}
\end{document}